\documentclass[12pt,tbtags,leqno]{amsart}

\usepackage{amssymb}
\usepackage{amsthm}
\usepackage{amsmath}
\usepackage{verbatim}

\usepackage{geometry}
\numberwithin{equation}{section}

\newcommand{\HH}{\mathcal{H}}

\newcommand{\HM}{\mathcal{M}}

\newcommand{\HL}{\mathcal{L}}

\newcommand{\HA}{\mathcal{A}}

\newcommand{\D}{\mathbb{D}}

\newcommand{\C}{\mathbb{C}}

\newcommand{\R}{\mathbb{R}}
\newcommand{\Z}{\mathbb{Z}}

\theoremstyle{plain}
\newtheorem{theorem}{Theorem}[section]
\newtheorem{lemma}[theorem]{Lemma}

\newtheorem{prop}[theorem]{Proposition}
\newtheorem{corollary}[theorem]{Corollary}

\theoremstyle{definition}

\begin{document}
\date{}

\author{Caixing Gu, Shuaibing Luo and Jie Xiao}
\address{Department of Mathematics, California Polytechnic State University, San Luis Obispo, CA 93407, USA}
\email{cgu@calpoly.edu}

\address{College of Mathematics and Econometrics, Hunan University, Changsha, Hunan, 410082, PR China}
\email{shuailuo2@126.com}

\address{Department of Mathematics and Statistics, Memorial University, St. John's, NL A1C5S7, Canada}
\email{jxiao@mun.ca}

\thanks{The second author is supported by the Young Teacher Program of Hunan University; the third author is supported by NSERC of Canada}

\title[Reducing subspaces for multiplication operators on the Dirichlet space]
{Reducing subspaces for multiplication operators on the Dirichlet space through local inverses and Riemann surface}

\begin{abstract}
This paper is devoted to the study of reducing subspaces for multiplication operator $M_\phi$ on the Dirichlet space with symbol of finite Blaschke product. The reducing subspaces of $M_\phi$ on the Dirichlet space and Bergman space are related. Our strategy is to use local inverses and  Riemann surface to study the reducing subspaces of $M_\phi$ on the Bergman space, and we discover a new way to study the Riemann surface for $\phi^{-1}\circ\phi$. By this means, we determine the reducing subspaces of $M_\phi$ on the Dirichlet space when the order of $\phi$ is $5$; $6$; $7$ and answer some questions of Douglas-Putinar-Wang \cite{DPW12}.
\medskip

\end{abstract}

\maketitle
\section*{Introduction}\label{s0}

Let $\D$ be the open unit disk in the complex plane and let $dA$ be the normalized area measure on the unit disk $\D$. For $\gamma\in [-1/2,1/2]$ let $\mathcal{D}_\gamma$ be the Dirichlet type space of all analytic functions $f$ on $\D$ with
$$
\int_{\D}|f'(z)|^2(1-|z|^2)^{1-2\gamma}\,dA(z)<\infty;
$$
see also \cite{St, X}. In particular, $\mathcal{D}_{-1/2}, \mathcal{D}_0$, and $\mathcal{D}_{1/2}$ are the Bergman space $B^2:=L^2_a(\D)$ of all analytic functions on $\D$ that are square integrable under $dA$, the Hardy space $H^2:=H^2(\D)$ of all analytic functions on $\D$ whose radial boundary values are square integrable under the arc-length on the unit circle $\mathbb T$, and the Dirichlet space $D^2:=D^2(\D)$ of analytic functions $f$ whose derivatives are square integrable under $dA$, respectively. Clearly,
$D^2\subset H^2\subset B^2$; this, along with the fact that $f\in D^2$ amounts to $f'\in B^2$, indicates that $B^2$ and $D^2$ may be viewed as two spaces symmetric about $H^2$.

If $\HH = B^2, H^2, D^2$, then a function $\phi$ is called a multiplier of $\HH$ if $\phi \HH \subseteq \HH$. Denote by $M(\HH)$ the multipliers of $\HH$. It follows from the closed graph theorem that every function $\phi \in M(\HH)$ induces a bounded linear operator $M_\phi: f \mapsto \phi f$ on $\HH$. The study of invariant subspaces or reducing subspaces for various operators on analytic function spaces has inspired much deep research. The structure of the invariant subspaces of the shift operator $M_z$ on $H^2, B^2$ and $D^2$ are different. On the Hardy space $H^2$, the famous Beurling theorem \cite{Be49} completely characterizes the invariant subspaces of $M_z$. On the Bergman space $B^2$, the lattice of invariant subspaces of $M_z$ is huge and its order structure is unknown \cite{BFP85}. On the Dirichlet space $D^2$, the invariant subspaces are studied in many papers, see e.g. \cite{BS84, EKR06, EKR09, RSa88, R911, RS91, RS92}, and there are open questions about the structure of the invariant subspaces of $D^2$ (\cite{EKR09}).

For a multiplier $\phi$ of $\HH$, a closed subspace $\HM \subseteq \HH$ is called a reducing subspace of $M_\phi$ provided that $\HM$ is invariant for both $M_\phi$ and $M_\phi^*$ (the adjoint operator of $M_\phi$). Let
$$
\varphi_\alpha(z) = \frac{\alpha- z}{1 - \overline{\alpha}z} \quad\forall\quad \alpha\in \D
$$
be the M\"{o}bius transformation sending $\alpha$ to the origin, and let $\psi = \varphi_{\alpha_1} \cdots \varphi_{\alpha_n}$ be the finite Blaschke product with zeros $\alpha_1, \cdots, \alpha_n \in \D$, then $\psi$ is a multiplier of $\HH$. The study of reducing subspaces of $M_\psi$ on $H^2$ was completed in the seventies (see \cite {Co78, Th76, Th77}). Investigation of reducing subspaces of $M_\psi$ on $B^2$ was started recently, see e.g. \cite{DSZ11, DPW12, GH11, GSZZ09, SlW98, SZZ10, Zh00}. See also \cite{Gu, Gu2, LZ, SL} for characterizations of reducing subspaces of $M_{\psi}$ for $\psi$ being a monomial of several variables on the Bergman space of polydisk or unit ball. The structure of reducing subspaces on $B^{2}$ was completely characterized in \cite{DPW12}. To be more precise, let $\mathsf{B}(B^2)$ be the class of all bounded linear operators on $B^2$ and
$$
\{M_\psi\}' = \{T \in \mathsf{B}(B^2): TM_\psi = M_\psi T\}
$$
be the commutant of $M_\psi$. Note that there is a one to one correspondence between the reducing subspaces of $M_\psi$ and the projections in $\{M_\psi\}'$, the problem of classifying the reducing subspaces of $M_\psi$ is reduced to studying the projections in $\{M_\psi\}'$. If
$$
\HA_\psi = \{M_\psi, M_\psi^*\}' \subset \mathsf{B}(B^2),
$$
then \cite{DPW12} says that $\HA_\psi$ is a commutative von Neumann algebra of dimension $q$, where $q$ is the number of connected components of the Riemann surface $S_\psi$ of $\psi^{-1}\circ \psi$. However, how to describe the reducing subspaces of $M_\psi$ on $B^2$ is challenging. The authors in \cite{DPW12} introduce the dual partition, then they use the dual partition to provide a indirect description of the reducing subspaces. But this description does not exhibit all the information of the reducing subspaces. In this paper, we refine their description by using local inverses of $\psi$ and Riemann surface $S_\psi$. By this means, we obtain more information of the reducing subspaces. For an open set $V \subseteq \D$, a local inverse of $\psi$ in $V$ is a function $\rho$ analytic in $V$ which satisfies $\rho(V) \subseteq \D$ and $\phi(\rho(z)) = \phi(z)$ on $V$. Note that the family of local inverses $\{\rho_0, \cdots, \rho_{n-1}\}$ has a group-like property under composition near the boundary of $\D$. This group-like property of $\{\rho_0, \cdots, \rho_{n-1}\}$, together with the property of the Riemann surface $S_\psi$ enables us to describe the reducing subspaces of $M_\psi$, see sections 2 and 6.

There are only a few results of the reducing subspaces of $M_\psi$ on $D^2$ (\cite{CL14, Zh09, Luo}). We need some definitions before we proceed. $M_\psi$ is called reducible if $M_\psi$ has a nontrivial reducing subspace. A reducing subspace $\HM$ of $M_\psi$ is called minimal if there is no nontrivial reducing subspace for $M_\psi$ contained in $\HM$. Two finite Blaschke products $\phi$ and $\psi$ are called equivalent if there are $w\in \D, |a| = 1$ such that $\psi = a \varphi_w \circ \phi$. Let $$
\widetilde{\HA_\psi} = \{M_\psi, M_\psi^*\}' \subset \mathsf{B}(D^2).
$$
Both \cite{Zh09} and \cite{CL14} showed that if the order of $\psi$ is $2$ and $\psi$ is not equivalent to $z^2$ then $M_\psi$ is irreducible on $D^2$, i.e. $\dim \widetilde{\HA_\psi} = 1$. Their results also reveal that the study of the reducing subspaces of $M_\psi$ on $D^2$ is difficult. One reason is that the inner product involves the derivatives which make the calculation complicated; the other reason is that $M_\psi$ is not a subnormal operator, we cannot use the theory of subnormal operators. In 2016, the second author in \cite{Luo} established a connection between the reducing subspaces on $D^2$ and on $B^2$, and studied the reducing subspaces of $M_\psi$ for no more than $4$ zeros. Let $U: D^2\rightarrow B^2$ be defined by $Uf = (zf)'$, then $U$ is a unitary operator. Luo \cite{Luo} showed that if $\HM$ is a reducing subspace of $M_\psi$ on $D^2$, then $U\HM$ is a reducing subspace of $M_\psi$ on $B^2$. This result is the key to investigate the reducing subspaces on $D^2$.

This paper is motivated by the idea in \cite{Co78, Th76, Th77, DSZ11} of using local inverses, analytic continuation and Riemann surface to study reducing subspaces. In this paper we determine the dimension of $\widetilde{\HA_\psi}$ on $D^2$ for general Blaschke product $\psi$. We then show that when the order $n$ of $\psi$ is $5$ or $7$, $M_\psi$ is reducible on $D^2$ if and only if $\psi$ is equivalent to $z^n$ - see Theorems \ref{rdopo5ubp1} \& \ref{rdop7m2p1}. We also give a description of the reducing subspaces of $M_\psi$ on $D^2$ for $\psi$ with $6$ zeros - see Theorem \ref{rdcwohsut}. In order to obtain those three theorems, we are led to not only deal with local inverses and von Neumann algebra - see Proposition \ref{vaiddatlms} and its Corollary \ref{irddpan}, but also handle partition and Riemann surface - see Propositions \ref{pezfpa}, \ref{ptwpztab}, \ref{ptfpezsabg} and \ref{ptwpstrs}, and Theorem \ref{urstdpopf}. Theorem \ref{urstdpopf} also answers a question in \cite{DPW12}.

The rest of this paper consists of the following six sections:

\begin{itemize}
\item[1.] Local inverses and von Neumann algebra;
\item[2.] Partition and Riemann surface;
\item[3.] Reducible $M_\phi$ on $D^2$ with $\phi$ being of order $5$;
\item[4.] Reducible $M_\phi$ on $D^2$ with $\phi$ being of order $6$;
\item[5.] Reducible $M_\phi$ on $D^2$ with $\phi$ being of order $7$;
\item[6.] Concluding remarks.
\end{itemize}

\section{Local inverses and von Neumann algebra}\label{s1}

We follow the description of the local inverses of $\phi$ in \cite{DPW12, DSZ11}. Suppose $\phi$ is a Blaschke product of order $n$. If
$$
E = \phi^{-1}(\phi(\{z \in D: \phi'(z) = 0\})),
$$
then $E$ is a finite set in $\D$. Such an $E$ is called the set of branched points of $\phi$. For an open set $V$ in $\D$, a local inverse of $\phi$ in $V$ is a function $f$ analytic in $V$ which satisfies $f(V) \subseteq \D$ and $\phi(f(z)) = \phi(z)$ for every $z$ in $V$. Note that for $z \in \D \setminus E$, $\phi$ is one to one in some neighborhood $V_{z_i}$ for each $z_i$ in
$$
\phi^{-1}\circ \phi(z) = \{z_0, \cdots, z_{n-1}\}.
$$
Therefore there are $n$ local inverses $\rho_0, \rho_1, \cdots, \rho_{n-1}$ for $\phi$ in $\D \setminus E$, i.e.
$$\phi^{-1}\circ \phi = \{\rho_0(z), \rho_1(z), \cdots, \rho_{n-1}(z)\}.
$$
Then each $\rho_i(z)$ is locally analytic and arbitrarily continuable in $\D \setminus E$. Suppose $\zeta_0$ is a fixed point on the boundary of $\D$, let $\Gamma$ be a curve in $\D$ passing through the points in $E$ and $\zeta_0$ so that $\D \backslash \Gamma$ is a simply connected region. As noted in \cite{DSZ11}, $n$ local inverses $\rho_0, \rho_1, \cdots, \rho_{n-1}$ for $\phi$ are well defined on $\D \backslash \Gamma$. We define an equivalence relation on the local inverses. We say that $\rho_i \thicksim \rho_j$ if there is a loop $\gamma$ in $\D \backslash E$ such that $\rho_i$ and $\rho_j$ are analytic continuation of each other along $\gamma$. Then $\thicksim$ is an equivalence relation. Using this equivalence relation, we obtain a partition
$$
\{G_1, G_2, \cdots, G_q\}\quad\hbox{for}\quad 1 < q \leq n\ \ \&\ \ \{\rho_0, \rho_1, \cdots, \rho_{n-1}\}.
$$
It is shown in \cite{DPW12} that $M_\phi$ has exact $q$ minimal reducing subspaces on $L^2_a(\D)$, where $q$ is the number of connected components of the Riemann surface $S_\phi$ of $\phi^{-1}\circ \phi$.

For $\HH = B^2$ or $D^2$ let
$$
\{M_\phi\}' = \{S \in \mathsf{B}(\HH):\ S M_\phi = M_\phi S\}
$$
be the commutant of $M_\phi$. Note that the reducing subspaces of $M_\phi$ are in one to one correspondence with the projections in $\{M_\phi\}'$, the problem of classifying the reducing subspaces of $M_\phi$ is equivalent to studying the projections in $\{M_\phi\}'$. Let
$$
\HA_\phi = \{M_\phi, M_\phi^*\}' \subset \mathsf{B}(B^2).
$$
Then \cite[Theorem 1.1]{DPW12} says that $\HA_\phi$ is a commutative von Neumann algebra of dimension $q$. In fact, suppose $\{G_1, G_2, \cdots, G_q\}$ is the partition for $\{\rho_0, \rho_1, \cdots, \rho_{n-1}\}$ and define
$$
\xi_i f (z)= \sum\limits_{\rho \in G_i} f(\rho(z)) \rho'(z)
\quad\forall\quad (z,f)\in (\D \setminus E)\times B^2.
$$
The result in \cite{DSZ11} asserts that $\xi_1, \cdots, \xi_q$ are bounded operators on $B^2$ which are linearly independent, and the von Neumann algebra $\HA_\phi$ is generated by $\xi_1, \cdots, \xi_q$.

If
$$
\widetilde{\HA_\phi} = \{M_\phi, M_\phi^*\}' \subset\mathsf{B}(D^2),
$$
then it is shown in \cite{Luo} that $\widetilde{\HA_\phi}$ is a commutative von Neumann algebra. To determine the dimension of $\widetilde{\HA_\phi}$, we need two lemmas.

\begin{lemma}\label{oivaibdc}
Let $(T,f)\in \widetilde{\HA_\phi}\times D^2$. Then there are $a_1, \cdots, a_q \in \C$ such that
$$
\begin{cases}
Tf (z)= \sum_{i=1}^q a_i \frac{F_i(z) - F_i(0)}{z};\\
T^* f (z) = \sum_{i=1}^q \overline{a_i} \frac{H_i(z) - H_i(0)}{z},
\end{cases}
$$
where
$$
\begin{cases}
F_i(z) = \sum\limits_{\rho \in G_i} f(\rho(z)) \rho(z);\\
H_i(z) = \sum\limits_{\rho \in G_i^{-1}} f(\rho(z)) \rho(z);\\
G_i^{-1} = \{\rho: \rho^{-1} \in G_i\}.
\end{cases}
$$
\end{lemma}

\begin{proof} First of all, we mention that for $f \in D^2$,
$$
F_i(z) = \sum_{\rho \in G_i} f(\rho(z)) \rho(z)
$$
is defined on $\D \setminus E$, where $E$ is the set of branched points of $\phi$. Since the continuation of any path in $\D \setminus E$ leads to a permutation in $\{\rho: \rho \in G_i\}$, $F_i(z)$ is unchanged under such a permutation and so is an analytic function well-defined on $\D \setminus E$ and analytically extends to the unit disk.

Next, let $U: D^2\rightarrow B^2$ be defined by
$$
Uf(z) = (zf)'(z)=\big(zf(z)\big)'.
$$
Then $U$ is a unitary operator. Let $P \in \widetilde{\HA_\phi}$ be a minimal projection, by Theorem 2.6 \cite{Luo}, $UPU^* \in A_\phi$. Note that $\widetilde{\HA_\phi}$ is a finite dimensional von Neumann algebra, all minimal projections span $\widetilde{\HA_\phi}$, therefore $UTU^* \in A_\phi$. Since $\HA_\phi = \text{span}\{\xi_1, \cdots, \xi_q\}$ (\cite{DSZ11}), there are $a_1, \cdots, a_q \in \C$ such that
$$
\begin{cases}
UTU^* = \sum_{i=1}^q a_i \xi_i;\\
UT^*U^* = \sum_{i=1}^q \overline{a_i} \xi_i^*.
\end{cases}
$$
This gives
$$
T = \sum_{i=1}^q a_i U^*\xi_iU,
$$
and consequently,
\begin{align*}
Tf & = \sum_{i=1}^q a_i U^*\xi_iU f\\
& = \sum_{i=1}^q a_i U^*\xi_i (zf)'\\
& = \sum_{i=1}^q a_i U^* \sum_{\rho \in G_i} (zf)'(\rho(z)) \rho'(z)\\
& = \sum_{i=1}^q a_i \frac{F_i(z) - F_i(0)}{z}.
\end{align*}
Note that (\cite{DSZ11})
$$
\xi_i^* f(z) = \sum\limits_{\rho \in G_i^{-1}} f(\rho(z)) \rho'(z).
$$
So we similarly have
$$
T^* f (z) = \sum_{i=1}^q \overline{a_i} \frac{H_i(z) - H_i(0)}{z}.
$$
\end{proof}

\begin{lemma}\label{oivgicbds}
Let $(T,f)\in\widetilde{\HA_\phi}\times D^2$. Suppose
$$
\begin{cases}
Tf (z)= \sum_{i=1}^q a_i \frac{F_i(z) - F_i(0)}{z};\\
T^* f (z) = \sum_{i=1}^q \overline{a_i} \frac{H_i(z) - H_i(0)}{z},
\end{cases}
$$
where
$$
\begin{cases}
a_i \in \C;\\
F_i(z) = \sum\limits_{\rho \in G_i} f(\rho(z)) \rho(z);\\
H_i(z) = \sum\limits_{\rho \in G_i^{-1}} f(\rho(z)) \rho(z).
\end{cases}
$$
Then
$$
\sum_{i=1}^q a_i F_i(0) = 0=\sum_{i=1}^q \overline{a_i} H_i(0).
$$
\end{lemma}
\begin{proof}
Since $T \in \widetilde{\HA_\phi}$, $TM_\phi f = M_\phi Tf$. Note that $\phi(\rho(z)) = \phi(z)$. So we have
$$
TM_\phi f(z) = \sum_{i=1}^q a_i \frac{\phi(z) F_i(z) - \phi(0)F_i(0)}{z}.
$$
It then follows from
$$
\begin{cases}
M_\phi Tf (z) = \sum_{i=1}^q a_i \phi(z) \frac{F_i(z) - F_i(0)}{z};\\
TM_\phi f = M_\phi Tf,
\end{cases}
$$
that
$$
\sum_{i=1}^q a_i F_i(0) \frac{\phi(z) - \phi(0)}{z} = 0.
$$
Thus
$$\sum_{i=1}^q a_i F_i(0) = 0.$$
Similarly,
$$
\sum_{i=1}^q \overline{a_i} H_i(0) = 0.
$$
\end{proof}

Let
$$
\HL = \text{span}\left\{(a_1, \cdots, a_q):\quad f \in D^2,\ \ \sum_{i=1}^q a_i F_i(0) = 0,\ \ \sum_{i=1}^q \overline{a_i} H_i(0) = 0\right\},
$$
where $F_i$ and $G_i$ are defined as above. Using Lemma \ref{oivaibdc}, we recover that $\widetilde{\HA_\phi}$ is a commutative von Neumann algebra as shown below.

\begin{prop}\label{vaiddatlms}
$\widetilde{\HA_\phi}$ is a commutative von Neumann algebra, and $\dim \widetilde{\HA_\phi} = \dim \HL$.
\end{prop}

\begin{proof}
Let $T \in \widetilde{\HA_\phi}$, then by the argument in Lemma \ref{oivaibdc}, there are $a_1, \cdots, a_q \in \C$ such that $T = \sum_{i=1}^q a_i U^*\xi_iU$. Since $\xi_i$ and $\xi_j$ commute for any $i ,j$, $\widetilde{\HA_\phi}$ is a commutative von Neumann algebra. Note that $\xi_1, \cdots, \xi_q$ are linearly independent, the rest follows from Lemmas \ref{oivaibdc} and \ref{oivgicbds}.
\end{proof}

The next result is presented in \cite{Luo}, but it can be easily obtained from Lemmas \ref{oivaibdc} and \ref{oivgicbds}.
\begin{corollary}\label{irddpan}
Let $\phi = \varphi_\alpha^n, \alpha \neq 0$. Then $M_\phi$ is irreducible on $D^2$.
\end{corollary}

\begin{proof}
Let $\omega = e^{\frac{2\pi i}{n}}$ be a primitive $n$-th root of unity, then
$$
\rho_j(z) = \varphi_\alpha(\omega^j \varphi_\alpha(z)),\ \ j = 0, 1, \cdots, n - 1,
$$
are $n$ local inverses for $\phi$. It follows that the partition for $\{\rho_0, \rho_1, \cdots, \rho_{n-1}\}$ is
$$\{\{\rho_0\}, \{\rho_1\}, \cdots, \{\rho_{n-1}\}\}.
$$
It is then clear that $\dim \widetilde{\HA_\phi} = \dim \HL = 1$, thus $M_\phi$ is irreducible on $D^2$.
\end{proof}

\section{Partitions and Riemann surfaces}\label{s2}

\subsection{Partitions for local inverses}\label{s21}

If $\rho_0, \rho_1, \cdots, \rho_{n-1}$ are $n$ local inverses for $\phi$, then there is an intrisinc order for these local inverses (\cite{DPW12}). In fact, let $A_s = \{z \in \C: s < |z| < 1\}$ for $0 < s < 1$, and let $\omega = e^{\frac{2\pi i}{n}}$ be a primitive $n$-th root of unity. By \cite[Lemma 2.1]{DPW12}, there exists an analytic function $u$ on a neighborhood of $\overline{\D} \setminus \gamma$ such that $\phi = u^n$, where $\gamma$ is an arc in $\D$ which contains the zero set of $\phi$. Moreover, there exists $s \in (0,1)$ such that $u : u^{-1}(\overline{A_s}) \rightarrow \overline{A_s}$ is invertible. By \cite[Lemma 2.2]{DPW12}, there exists a family of local inverses for $\phi$ on $\Omega = u^{-1}(A_s)$. We then label the local inverses $\{\rho_0, \rho_1, \cdots, \rho_{n-1}\}$ such that $\rho_i(z) = u^{-1}(\omega^iu(z))$ on $\Omega$ for $0 \leq i \leq n-1$. Then each $\rho_i$ is invertible on $\Omega$ and for any $\rho_i, \rho_j, z \in \Omega$,
$$
\rho_i \circ \rho_j(z) = \rho_{i+j~ \text{mod} ~n}(z) = \rho_j \circ \rho_i(z),
$$
i.e. $\{\rho_0, \rho_1, \cdots, \rho_{n-1}\}$ has a group-like property under composition on $\Omega$.

We remark that by the invertibility of $u : u^{-1}(\overline{A_s}) \rightarrow \overline{A_s}$, we can also get that the operators $\xi_i: B^2 \rightarrow B^2$ defined by
$$\xi_i f(z) = \sum_{\rho \in G_i} f(\rho(z)) \rho'(z)
$$ are bounded. In fact, there exists a constant $C$ such that for $f \in B^2$,
$$\int_{\D} |f^2(z)| dA(z) \leq C \int_\Omega |f^2(z)| dA(z),$$
thus
\begin{align*}
\int_{\D} |\xi_if|^2 dA &\leq C \int_\Omega |\xi_if|^2 dA\\
& \leq \widetilde{C}\sum_{\rho \in G_i} \int_\Omega |f(\rho(z)) \rho'(z)|^2 dA(z)\\
& = \widetilde{C} \# G_i \int_\Omega |f|^2 dA\\
& \leq \widetilde{C}\# G_i \int_{\D} |f|^2 dA,
\end{align*}
where $\# G_i$ denotes the number of elements in $G_i$.

For a computational purpose, we write $j \in G_k$ if $\rho_j \in G_k$, so
$$
G_k = \{j_1, \cdots, j_l\}\quad\hbox{as}\quad G_k = \{\rho_{j_1}, \cdots, \rho_{j_l}\}.
$$
Since $\{\rho_0, \rho_1, \cdots, \rho_{n-1}\}$ has a group-like property, $\{G_1, G_2, \cdots, G_q\}$ is a partition of the additive group $\Z_n = \{0, 1, \cdots, n-1\}$. We now define the dual partition. For integers $j_1, j_2 \in \{0, 1, \cdots, n\}$, we write $j_1 \thicksim j_2$, if $\sum_{\rho_i \in G_k} \omega^{ij_1} = \sum_{\rho_i \in G_k} \omega^{ij_2}$. This equivalence relation then partition $\{0, 1, \cdots, n\}$ into equivalent classes $\{G_1', \cdots, G_l'\}$. We call this partition the dual partition. By the discussion in \cite{DPW12}, we have the following necessary conditions for the partitions $\{G_1, G_2, \cdots, G_q\}$.
\begin{itemize}
\item[{\rm ($A_1$)}] One of $\{G_k\}$ is $\{0\}$ since $\rho_0(z) = z$.

\item[{\rm ($A_2$)}] By \cite[Lemma 7.4]{DSZ11}, for each $G_j = \{j_1, \cdots, j_m\}$, there exists $k$ such that
$$G_k = G_j^{-1} = \{n - j_1, \cdots, n - j_m\}.$$

\item[{\rm ($A_3$)}] By \cite[Theorem 7.6]{DSZ11}, for any $G_j, G_k$, there are $G_{l_1}, \cdots, G_{l_m}$ such that
\begin{align*}
G_j + G_k = G_{l_1} \cup \cdots \cup G_{l_m} \quad ~\text{counting multiplicities on both sides}.
\end{align*}

\item[{\rm ($A_4$)}] By \cite[Corollary 3.10]{DPW12}, the dual partition also has $q$ elements, i.e. $l = q$.
\end{itemize}
We will see that these necessary conditions are not sufficient, and we are led to study Riemann surface.

Now we use the above conditions to obtain the possible partitions $\{G_1, G_2, \cdots, G_q\}$ when the order of $\phi=u^n$ is $5$. Accordingly, if $n=5$, then \cite[Corollary 8.4]{DSZ11} implies $q\neq 4$, and hence we have the following cases:

(i)  If $q = 5$, then the partition is $\{G_1, G_2, \cdots, G_q\}=\{\{0\}, \{1\}, \{2\}, \{3\}, \{4\}\}$.

(ii) If $q = 3$, without loss of generality, suppose $G_1 = \{0\}$. Let $m = \min\{\#G_2, \# G_3\}$. By condition ($A_2$), $m$ can not be $1$. Thus $m = 2$, then $\#G_2 = \#G_3 =2$, and there are essentially three cases.

(a) $G_2 = \{1, 2\}, G_3 = \{3, 4\}$;

(b) $G_2 = \{1, 3\}, G_3 = \{2, 4\}$;

(c) $G_2 = \{1, 4\}, G_3 = \{2, 3\}$.

\medskip

Case (a) doesn't satisfy condition ($A_3$), since $G_2 + G_2 = \{2, 3, 3, 4\}$. Similarly, case (b) doesn't satisfy condition ($A_3$). So we have the possible partition $(c)$.

(iii) If $q = 2$, then the partition is $\{\{0\}, \{1, 2, 3, 4\}\}$.

Therefore when $n = 5$, the possible partitions are
$$
\begin{cases}
\{\{0\}, \{1\}, \{2\}, \{3\}, \{4\}\}$, $\{\{0\}, \{1, 4\}, \{2, 3\}\};\\ \{\{0\}, \{1, 2, 3, 4\}\}.
\end{cases}
$$
We note here that if $q = n$, then each $\rho_i$ is equivalent to itself, thus each $\rho_i$ extends analytic to the unit disc and has modulus 1 on the unit circle, so each $\rho_i$ is a M\"{o}bius transform. Hence $\phi$ is equivalent to $\varphi_\alpha^n, \alpha \in \D$. We are about to show that for most Blaschke products with $5$ zeros, the partition is $\{\{0\}, \{1, 2, 3, 4\}\}$.

\subsection{Riemann surfaces}\label{s22}

Let $\phi ={P}/{Q}$ be a finite Blaschke product of order $n$, where $P$ and $Q$ are two polynomials of degree less than or equal to $n$. Let
$$f(w, z) = P(w)Q(z) - P(z) Q(w).$$
Then $f(w,z)$ is a polynomial of $w$ with degree $n$, and the coefficients are polynomials of $z$ with degree less than or equal to $n$. In the ring $\C[z,w]$, we factor
$$f(w,z) = \prod_{i=1}^q p_i(w,z)^{n_i},$$
where $p_i(w,z)$ are irreducible polynomials. Note that on $\D^2$, $\phi(w) - \phi(z) = 0$ if and only if $f(w,z) = 0$. Since Bochner's theorem (\cite{Wa18}) says that $\phi$ has exact $n - 1$ critical points in $\D$, it follows that (see \cite{DSZ11})
$$f(w,z) = \prod_{i=1}^q p_i(w,z).$$
Theorem 3.1 \cite{DSZ11} says that the number of connected components of the Riemann surface $S_\phi$ equals the number of irreducible factors of $f(w,z)$.
\begin{theorem}[\cite{DSZ11}]\label{ccrsrsbd}
Let $\phi$ be a Blaschke product of order $n$ and $f(w,z) = \Pi_{j=1}^q p_j(w,z)$. Suppose that $p(w,z)$ is one of factors of $f(w,z)$. Then the Riemann surface $S_p$ is connected if and only if $p(w,z)$ is irreducible. Hence $q$ equals the number of connected components of the Riemann surface $S_\phi = S_f$.
\end{theorem}
Note that $f(w,z) = 0$ is an algebraic equation with analytic coefficients in $z \in \D$, the above theorem also follows from Theorem 8.9 \cite{Fo91}.

For two finite Blaschke products $\phi$ and $\varphi$, we say that $\phi$ is equivalent to $\varphi$ provided that there are $|a| = 1, \alpha \in \D$ such that $\varphi_\alpha(\phi(z)) = a \varphi(z)$. For two equivalent Blaschke products $\phi$ and $\varphi$, $M_\phi$ and $M_\varphi$ are functional calculus of each other, and hence have the same reducing subspaces.

Notice that for a finite Blaschke product $\phi$, if $\phi'(\alpha) = 0$ and $\lambda = \phi(\alpha)$, then
$$
\varphi_\lambda (\phi(z)) = a \varphi_\alpha^2(z) \psi(z),
$$
where $|a| = 1, \psi$ is a finite Blaschke product of order $n - 2$, thus $\phi$ is equivalent to $\varphi_\alpha^2 \psi$. Let $U_\alpha: B^2 \rightarrow B^2$ be defined by
$$
U_\alpha g = g(\varphi_\alpha) k_\alpha,
$$
where
$$
k_\alpha(z) = \frac{1-|\alpha|^2}{(1-\overline{\alpha}z)^2}
$$
is the normalized reproducing kernel for $B^2$, then $U_\alpha$ is a unitary operator, and
$$
U_\alpha^* M_\phi U_\alpha = M_{\phi \circ \varphi_\alpha}.
$$
Therefore, if the order of $\phi$ is $5$, we only need to study the reducing subspaces of $M_\phi$ on $B^2$ when
$$
\phi = z^4 \varphi_\alpha, z^3\varphi_\alpha\varphi_\beta\quad\hbox{or}\quad z^2\varphi_\alpha\varphi_\beta\varphi_\gamma, \alpha, \beta, \gamma \in \D \setminus \{0\}.
$$
Theorem \ref{ccrsrsbd} plays an important role here. Note that for a finite Blaschke product of order $n$, if the partition for $\phi$ is $\{\{0\}, \{1, 2, \cdots, n - 1\}\}$, then $M_\phi$ has exact two minimal reducing subspaces on $B^2$: $M_0(\phi), M_0(\phi)^\perp$, where
$$M_0(\phi) = \text{span}\{\phi'\phi^j: j \geq 0\}
$$ is always a minimal reducing subspace of $M_\phi$ on $B^2$.

\begin{prop}\label{pezfpa}
Let $\phi = z^4 \varphi_\alpha, \alpha \in \D\backslash \{0\}$. Then the partition is $\{\{0\}, \{1, 2, 3, 4\}\}$.
\end{prop}
\begin{proof}
We prove it by contradiction. If the partition is not $\{\{0\}, \{1, 2, 3, 4\}\}$, then the partition is $\{\{0\}, \{1, 4\}, \{2, 3\}\}$. Note that $\phi^{-1} \circ \phi(0) = \{0, 0, 0, 0, \alpha\}$. So, without loss of generality we may assume
$$
\{\rho_1(0), \rho_4(0)\} = \{0, 0\}\quad\&\quad \{\rho_2(0), \rho_3(0)\} = \{0, \alpha\}.
$$
By Theorem \ref{ccrsrsbd}, $\{\rho_1, \rho_4\}$ forms one component of the Riemann surface $S_\phi$, so if
$$
f(z) = \rho_1(z) + \rho_4(z)
$$
then $f$ is analytic on $\D \setminus E$ and bounded, thus $f$ is bounded analytic on $\D$, i.e., $f\in H^\infty(\D)$. Similarly,
$$
g(z) = f(\rho_1(z)) + f(\rho_4(z))
$$
is bounded analytic on $\D$, i.e., $g\in H^\infty(\D)$. Recall that $\{\rho_i\}$ has a group-like property on $\Omega = u^{-1}(A_s)$, we get
$$g(z) = \rho_2(z) + \rho_0(z) + \rho_0(z) + \rho_3(z) = \rho_2(z) + \rho_3(z) + 2z\quad\forall\quad z \in \Omega.$$
Since $\{\rho_2, \rho_3\}$ forms another component of the Riemann surface $S_\phi$, $\rho_2 + \rho_3$ is analytic on $\D\setminus E$, and so in $H^\infty(\D)$. This in turn implies
$$g(z) = f(\rho_1(z)) + f(\rho_4(z)) = \rho_2(z) + \rho_3(z) + 2z\quad\forall\quad z \in \D.$$
Therefore
$$g(0) = f(\rho_1(0)) + f(\rho_4(0)) = \rho_2(0) + \rho_3(0) + 0 = \alpha,$$
but $$f(\rho_1(0)) + f(\rho_4(0)) = f(0) + f(0) = 0 \neq \alpha.$$
This is a contradiction, hence the partition $\{\{0\}, \{1, 4\}, \{2, 3\}\}$ is impossible, so the partition is $\{\{0\}, \{1, 2, 3, 4\}\}$.
\end{proof}
Note that $M_{z^4\varphi_\alpha}$ is unitarily equivalent to $M_{z\varphi_\alpha^4}$. So, it follows that if $\phi = z \varphi_\alpha^4, \alpha \neq 0$, then the partition is also $\{\{0\}, \{1, 2, 3, 4\}\}$.

\begin{prop}\label{ptwpztab}
Let $\phi = z^3 \varphi_\alpha \varphi_\beta, \alpha, \beta \in \D\backslash \{0\}$. Then the partition is $\{\{0\}, \{1, 2, 3, 4\}\}$.
\end{prop}
\begin{proof}
We prove it by contradiction. We have the following two cases.

(a) $\alpha \neq \beta$. If the partition is not $\{\{0\}, \{1, 2, 3, 4\}\}$, then it is $\{\{0\}, \{1, 4\}, \{2, 3\}\}$. Note that $\phi^{-1} \circ \phi(0) = \{0, 0, 0, \alpha, \beta\}$. So
$$
\{\rho_1(0), \rho_4(0)\} = \{0, 0\}, \{0, \alpha\}, \{0, \beta\}\quad\hbox{or}\quad \{\alpha, \beta\}.
$$

Case I: If $\{\rho_1(0), \rho_4(0)\} = \{0, 0\}$, then $\{\rho_2(0), \rho_3(0)\} = \{\alpha, \beta\}$. Let
$$
f(z) = \rho_1(z) + \rho_4(z), g(z) = f(\rho_1(z))  f(\rho_4(z)).
$$
As in Proposition \ref{pezfpa}, we can get
$$
f, g \in H^\infty(\D)\quad\&\quad g(z) = [\rho_2(z) + z]  [\rho_3(z) + z],
$$
whence finding
$$
g(0) = f(\rho_1(0))  f(\rho_4(0)) = \alpha\beta = f(0)  f(0) = 0
$$
which is a contradiction.

Case II: If $\{\rho_1(0), \rho_4(0)\}= \{\alpha, \beta\}\ \& \ \{\rho_2(0), \rho_3(0)\}= \{0, 0\}$, this is essentially Case I.

Case III: If $\{\rho_1(0), \rho_4(0)\} = \{0, \alpha\}\ \&\ \{\rho_2(0), \rho_3(0)\} = \{0, \beta\}$, then letting
$$
\begin{cases}
f(z) = \rho_1(z) + \rho_4(z);\\
g_1(z) = f(\rho_1(z)) + f(\rho_4(z)) = \rho_2(z) + \rho_3(z) + 2z,
\end{cases}
$$
gives
$$
g_1(0) = \beta = f(0) + f(\alpha) = \alpha + f(\alpha).
$$
Note that
$$
\phi^{-1} \circ \phi(\alpha) = \{0, 0, 0, \alpha, \beta\}\ \& \ \rho_0(\alpha) = \alpha.
$$
Thus we have $f(\alpha) = 0$ or $\beta$. Since $\alpha \neq \beta\ \&\ \alpha\beta \neq 0$, we cannot get $\alpha + f(\alpha) = \beta$, thus we have a contradiction.

Case IV: If $\{\rho_1(0), \rho_4(0)\} = \{0, \beta\}\ \&\ \{\rho_2(0), \rho_3(0)\} = \{0, \alpha\}$, this is essentially Case III.

Therefore, if $\phi = z^3 \varphi_\alpha \varphi_\beta, \alpha \neq \beta, \alpha\beta \neq 0$, then the partition is $\{\{0\}, \{1, 2, 3, 4\}\}$.

(b) $\alpha = \beta$. If the partition is not $\{\{0\}, \{1, 2, 3, 4\}\}$, then it is $\{\{0\}, \{1, 4\}, \{2, 3\}\}$. Note that $\phi^{-1} \circ \phi(0) = \{0, 0, 0, \alpha, \alpha\}$, there are essentially two cases: $\{\rho_1(0), \rho_4(0)\} = \{0, 0\}$ or $\{0, \alpha\}$.

Case V: If $\{\rho_1(0), \rho_4(0)\} = \{0, 0\}$, then $\{\rho_2(0), \rho_3(0)\}= \{\alpha, \alpha\}$. Let
$$
\begin{cases}
f(z) = \rho_1(z) + \rho_4(z);\\
g(z) = f(\rho_1(z)) + f(\rho_4(z)) = \rho_2(z) + \rho_3(z) + 2z.
\end{cases}
$$
Then $g(0) = 2f(0) = 2\alpha$, and hence $f(0) = \alpha$, but $f(0) = 0$, this is a contradiction.

Case VI: If $\{\rho_1(0), \rho_4(0)\} = \{0, \alpha\}$, then $\{\rho_2(0), \rho_3(0)\}= \{0, \alpha\}$. If $f, g$ are as in Case V, then
$$
g(0) = f(0) + f(\alpha) = \alpha.
$$
Since $f(0) = \alpha$, $f(\alpha) = 0$. Upon noticing
$$
\phi^{-1} \circ \phi(\alpha) = \{0, 0, 0, \alpha, \alpha\}\quad\&\quad \rho_0(\alpha) = \alpha,
$$
we have
$$
\{\rho_1(\alpha), \rho_4(\alpha)\} = \{0, 0\}\ \&\ \{\rho_2(\alpha), \rho_3(\alpha)\} = \{0, \alpha\}.
$$
thereby getting $g(\alpha) = 2f(0) = 3\alpha$ - this is a contradiction. Hence the partition is $\{\{0\}, \{1, 2, 3, 4\}\}$. The proof is complete.
\end{proof}

Note that $M_{z^3\varphi_\alpha^2}$ is unitarily equivalent to $M_{z^2\varphi_\alpha^3}$, it follows that if $\phi = z^2 \varphi_\alpha^3$ then the partition is also $\{\{0\}, \{1, 2, 3, 4\}\}$.

\begin{prop}\label{ptfpezsabg}
Let $\phi = z^2 \varphi_\alpha \varphi_\beta \varphi_\gamma$ with $\alpha, \beta, \gamma$ being mutually distinct and $\alpha\beta\gamma \neq 0$. Then the partition is $\{\{0\}, \{1, 2, 3, 4\}\}$.
\end{prop}
\begin{proof}
We prove it by contradiction. If the partition is not $\{\{0\}, \{1, 2, 3, 4\}\}$, then it is $\{\{0\}, \{1, 4\}, \{2, 3\}\}$. Without loss of generality, suppose
$$
\begin{cases}
\{\rho_1(0), \rho_4(0)\} = \{0, \alpha\};\\
\{\rho_2(0), \rho_3(0)\} = \{\beta, \gamma\}.
\end{cases}
$$
If
$$
\begin{cases}
f(z) = \rho_1(z) + \rho_4(z);\\
g(z) = f(\rho_1(z)) + f(\rho_4(z)) = \rho_2(z) + \rho_3(z) + 2z,
\end{cases}
$$
then
$$
g(0)  = f(0) + f(\alpha) = \beta + \gamma = \alpha + f(\alpha).
$$
Note that
$$
\phi^{-1} \circ \phi(\alpha) = \{0, 0, \alpha, \beta, \gamma\}\ \&\ \rho_0(\alpha) = \alpha.
$$
So, we have $f(\alpha) = 0, \beta, \gamma$ or $\beta+\gamma$.

If $f(\alpha) = \beta + \gamma$, then $\alpha = 0$ - this is a contradiction. If $f(\alpha) = \gamma$, then $\alpha = \beta$ - this is also a contradiction. Similarly, if $f(\alpha) = \beta$, then $\alpha = \gamma$, which is also a contradiction. If $f(\alpha) = 0$, then
$$
\begin{cases}
\alpha = \beta + \gamma;\\
\{\rho_1(\alpha), \rho_4(\alpha)\} = \{0, 0\};\\
\{\rho_2(\alpha), \rho_3(\alpha)\} = \{\beta, \gamma\}.
\end{cases}
$$
So
$$
g(\alpha) = 2f(0) = 2\alpha + \beta + \gamma.
$$
Since $f(0) = \alpha$, it follows that $\beta + \gamma = \alpha = 0$ - this is a contradiction. Thus the partition is $\{\{0\}, \{1, 2, 3, 4\}\}$.
\end{proof}

Before discussing the partition for $\phi = z^2 \varphi_\alpha^2 \varphi_\beta, \alpha \neq \beta, \alpha\beta \in \D\backslash\{0\}$, we need two more lemmas.

\begin{lemma}\label{farffud}
Let
$$
\begin{cases}
\phi = z^2\varphi_\alpha\varphi_\beta\varphi_\gamma = \frac{P(z)}{Q(z)};\\
\alpha, \beta, \gamma \in \D;\\
f(w,z) =  P(w)Q(z) - P(z)Q(w).
\end{cases}
$$
If
$$f(w,z) = (w-z)[d_0(z)w^4 + d_1(z) w^3 + d_2(z) w^2 + d_3(z) w + d_4(z)],$$
then
$$
\begin{cases}
d_4(z) = z(\alpha - z) (\beta-z) (\gamma - z);\\
d_3(z) = (\alpha - z) (\beta-z) (\gamma - z) [1-(\overline{\alpha} + \overline{\beta} + \overline{\gamma})z];\\
d_2(z)= (1- \overline{\alpha}z) (1- \overline{\beta}z)(1-\overline{\gamma}z) [-(\alpha\beta+\beta\gamma+\alpha\gamma)+(\alpha+\beta+\gamma)z - z^2]\\
\quad\quad\quad +\ \overline{\alpha} \overline{\beta} \overline{\gamma} z^2(\alpha - z) (\beta-z) (\gamma - z);\\
d_1(z) = (1- \overline{\alpha}z) (1- \overline{\beta}z)(1-\overline{\gamma}z) [(\alpha+\beta+\gamma) - z];\\
d_0(z) = - (1- \overline{\alpha}z) (1- \overline{\beta}z)(1-\overline{\gamma}z).$$
\end{cases}
$$
\end{lemma}
\begin{proof}
Since
\begin{align*}
f(w,z)& = P(w)Q(z) - P(z)Q(w) \\
& = w^2 (\alpha - w) (\beta-w) (\gamma - w)(1- \overline{\alpha}z) (1- \overline{\beta}z)(1-\overline{\gamma}z) \\
&\hspace{0.5cm}- z^2 (\alpha - z) (\beta-z) (\gamma - z)(1- \overline{\alpha}w) (1- \overline{\beta}w)(1-\overline{\gamma}w),
\end{align*}
we have the following five equations:
$$
f(0,z) = -z^2 (\alpha - z) (\beta-z) (\gamma - z);
$$
$$
\frac{\partial f}{\partial w}(0,z) = -z^2 (\alpha - z) (\beta-z) (\gamma - z) [-(\overline{\alpha} + \overline{\beta} + \overline{\gamma})];
$$
$$
\frac{\partial^2 f}{\partial w^2}(0,z) = 2 \alpha\beta\gamma (1- \overline{\alpha}z) (1- \overline{\beta}z)(1-\overline{\gamma}z) - 2z^2 (\alpha - z) (\beta-z) (\gamma - z) (\overline{\alpha\beta} + \overline{\beta\gamma} + \overline{\alpha\gamma});
$$
$$
\frac{\partial^3 f}{\partial w^3}(0,z) = -6 (\alpha\beta+\beta\gamma+\alpha\gamma) (1- \overline{\alpha}z) (1- \overline{\beta}z)(1-\overline{\gamma}z) - z^2 (\alpha - z) (\beta-z) (\gamma - z) (-6\overline{\alpha} \overline{\beta} \overline{\gamma});
$$
$$
\frac{\partial^5 f}{\partial w^5}(0,z) = -5! (1- \overline{\alpha}z) (1- \overline{\beta}z)(1-\overline{\gamma}z).
$$
Note that
$$
\begin{cases}
f(0,z) = - zd_4(z);\\
\frac{\partial f}{\partial w}(0,z) = -zd_3(z) + d_4(z);\\
\frac{\partial^2 f}{\partial w^2}(0,z) = -2zd_2(z) + 2d_3(z);\\
\frac{\partial^3 f}{\partial w^3}(0,z) = - 6zd_1(z) + 6d_2(z);\\ \frac{1}{5!}\frac{\partial^5 f}{\partial w^5}(0,z) = d_0(z).
\end{cases}
$$
 The conclusion follows from the equations of the derivatives of $f$.
\end{proof}

\begin{lemma}\label{psptwpds}
Let $\phi = z^2 \varphi_\alpha^2 \varphi_\beta, \alpha \neq \beta, \alpha\beta \neq 0$. If the partition for $\phi$ is $\{\{0\}, \{1, 4\}, \{2, 3\}\}$, then
$$
\begin{cases}
\{\rho_1(0), \rho_4(0)\} = \{0, \alpha\};\\
\{\rho_1(\alpha), \rho_4(\alpha)) = \{0, \beta\};\\
\{\rho_1(\beta), \rho_4(\beta)) = \{\alpha, \alpha\},
\end{cases}
$$
and
$$
\begin{cases}
\{\rho_2(0), \rho_3(0)\} = \{\alpha, \beta\};\\
\{\rho_2(\alpha), \rho_3(\alpha)\} = \{0, \alpha\};\\
\{\rho_2(\beta), \rho_3(\beta)\} = \{0, 0\}.
\end{cases}
$$

\end{lemma}
\begin{proof}
Suppose the partition for $\phi$ is $\{\{0\}, \{1, 4\}, \{2, 3\}\}$. If
$$
f(z) = \rho_1(z) + \rho_4(z), g(z) = f(\rho_1(z)) + f(\rho_4(z)),
$$
then
$$
g(z) = \rho_2(z) + \rho_3(z) + 2z.
$$
There are essentially two cases:
$$
\{\rho_1(0), \rho_4(0)\} = \{0, \beta\}\ \hbox{or}\ \{\rho_1(0), \rho_4(0)\} = \{0, \alpha\}.
$$

If
$$
\{\rho_1(0), \rho_4(0)\} = \{0, \beta\}\ \& \ \{\rho_2(0), \rho_3(0)\} = \{\alpha, \alpha\},
$$
then
$$g(0) = 2\alpha = f(0) + f(\beta) = \beta + f(\beta).$$
Note that
$$
\phi^{-1}\circ \phi(\beta) = \{0, 0, \alpha, \alpha, \beta\}\ \&\ \rho_0(\beta) = \beta.
$$
So we have $f(\beta) = 0, 2\alpha$ or $\alpha$. If $f(\beta) = \alpha$, then $\alpha = \beta$, this is a contradiction. If $f(\beta) = 2 \alpha$, then $\beta = 0$ - this is also a contradiction. If $f(\beta) = 0$, then
$$
\begin{cases}
\beta = 2\alpha;\\
\{\rho_1(\beta), \rho_4(\beta)\} = \{0, 0\};\\
\{\rho_2(\beta), \rho_3(\beta)\} = \{\alpha, \alpha\},
\end{cases}
$$
and hence
$$
g(\beta) = 2\alpha + 2\beta = 2f(0) = 2\beta,
$$
so $\alpha = 0$ - this is a contradiction. Consequently, we have
$$
\{\rho_1(0), \rho_4(0)\} = \{0, \alpha\}\ \&\ \{\rho_2(0), \rho_3(0)\} = \{\alpha, \beta\}.
$$
Now, let $f, g$ be the same as above. Then
$$g(0) = \alpha + \beta = f(0) + f(\alpha) = \alpha + f(\alpha).$$
This in turn implies
$$
\begin{cases}
f(\alpha) = \beta;\\
\{\rho_1(\alpha), \rho_4(\alpha)\} = \{0, \beta\}, \{\rho_2(\alpha), \rho_3(\alpha)\} = \{0, \alpha\}.
\end{cases}
$$
Then
$$g(\alpha) = \alpha + 2\alpha = f(0) + f(\beta) = \alpha + f(\beta),$$
and hence
$$
\begin{cases}
f(\beta) = 2\alpha;\\
\{\rho_1(\beta), \rho_4(\beta)\} = \{\alpha, \alpha\};\\
\{\rho_2(\beta), \rho_3(\beta)\} = \{0, 0\}.
\end{cases}
$$
This finishes the proof.
\end{proof}
Now we can study the partition for $\phi = z^2 \varphi_\alpha^2 \varphi_\beta, \alpha \neq \beta, \alpha\beta \neq 0$.

\begin{prop}\label{ptwpstrs}
Let $\phi = z^2 \varphi_\alpha^2 \varphi_\beta, \alpha \neq \beta, \alpha\beta \neq 0$. Then the partition for $\phi$ is $$\{\{0\}, \{1, 4\}, \{2, 3\}\}$$ if and only if $\alpha/\beta \in \R$ and $\varphi_\beta(\alpha) = \frac{\alpha^2}{\beta}$.
\end{prop}
\begin{proof}

Suppose $\phi = z^2 \varphi_\alpha^2 \varphi_\beta, \alpha \neq \beta, \alpha\beta \neq 0$, by Lemma \ref{farffud},
\begin{align*}
f(w,z) &= (w-z)[d_0(z)w^4 + d_1(z) w^3 + d_2(z) w^2 + d_3(z) w + d_4(z)]\\ & = (w-z) \Big(- (1- \overline{\alpha}z)^2 (1- \overline{\beta}z) w^4 + (1- \overline{\alpha}z)^2 (1- \overline{\beta}z) [(2\alpha + \beta) -z] w^3\\
&\hspace{0.2cm} + \{ (1- \overline{\alpha}z)^2 (1- \overline{\beta}z)[-(2\alpha\beta + \alpha^2) + (2\alpha + \beta) z - z^2] + \overline{\alpha}^2 \overline{\beta} z^2 (\alpha-z)^2 (\beta - z)\}w^2\\
&\hspace{0.2cm} + (\alpha-z)^2 (\beta - z) [1- (2\overline{\alpha}+\overline{\beta})z]w + z (\alpha-z)^2 (\beta - z) \Big).
\end{align*}

If the partition for $\phi$ is $\{\{0\}, \{1, 4\}, \{2, 3\}\}$, then by Theorem \ref{ccrsrsbd}, there are two irreducible polynomials $p_1(w,z)$ and $p_2(w,z)$ such that $f(w,z) = (w-z)p_1(w,z) p_2(w,z)$. Suppose $p_1(w,z) = a_0(z) w^2 + a_1(z) w + a_2(z), p_2(w,z) = b_0(z) w^2 + b_1(z) w + b_2(z)$. Note that $\{\rho_1, \rho_4\}$ and $\{\rho_2, \rho_3\}$ form two components of $S_\phi = S_f$, and the positions for $w$ and $z$ are symmetric, we have $a_0(z), b_0(z)$ are not $0$, and $\text{deg} a_i(z) \leq 2, \text{deg} b_i(z) \leq 2$.

Multiplying $p_1(w,z)$ with $p_2(w,z)$ and comparing the coefficients of $w^j$ with $\frac{f(w,z)}{w-z}$, $j = 0, 1, 2, 3, 4$, we obtain
\begin{align}
&a_0(z) b_0(z) = - (1- \overline{\alpha}z)^2 (1- \overline{\beta}z), \label{aobo}\\
&a_0(z) b_1(z) + a_1(z) b_0(z) = (1- \overline{\alpha}z)^2 (1- \overline{\beta}z) [(2\alpha + \beta) -z], \label{aob1}\\
&a_0(z) b_2(z) + a_1(z) b_1(z) + a_2(z) b_0(z)= (1- \overline{\alpha}z)^2 (1- \overline{\beta}z) \label{aob2}\\
&\hspace{0.5cm} [-(2\alpha\beta + \alpha^2) + (2\alpha + \beta) z - z^2] + \overline{\alpha}^2 \overline{\beta} z^2 (\alpha-z)^2 (\beta - z), \notag\\
&a_1(z) b_2(z) + a_2(z) b_1(z) =  (\alpha-z)^2 (\beta - z) [1- (2\overline{\alpha}+\overline{\beta})z], \label{a1b2}\\
& a_2(z) b_2(z) = z (\alpha-z)^2 (\beta - z). \label{a2b2}
\end{align}
Note that $\{\rho_1(z) \rho_4(z), \rho_2(z) \rho_3(z)\} = \{\frac{a_2(z)}{a_0(z)}, \frac{b_2(z)}{b_0(z)}\}$, and $\rho_1(z) \rho_4(z) \neq \rho_2(z) \rho_3(z)$. By Lemma \ref{psptwpds}, we have $\{\rho_1(\alpha), \rho_4(\alpha)\} = \{0, \beta\}, \{\rho_2(\alpha), \rho_3(\alpha)\} = \{0, \alpha\}$, thus $a_2(\alpha) = b_2(\alpha) = 0$, and so $\alpha - z$ is a factor of $a_2(z)$ and $b_2(z)$. There are essentially two cases.

Case I: $a_0(z) = (1- \overline{\alpha}z)^2$, $b_0(z) = \overline{\beta}z - 1$. Then by (\ref{aob1}), $(1 - \overline{\alpha}z)^2$ is a factor of $a_1(z)$, $\overline{\beta}z - 1$ is a factor of $b_1(z)$. Thus $a_1(z) = c_1 (1 - \overline{\alpha}z)^2$.

(a) If $a_2(z) = c_2 (\alpha - z)^2, b_2(z) = \frac{1}{c_2} (\beta - z)z$, this contradicts the fact that $\alpha - z$ is a factor of $a_2(z)$ and $b_2(z)$.

(b) If $a_2(z) = c_2 (\beta - z)z, b_2(z) = \frac{1}{c_2} (\alpha - z)^2$, similarly, this is also a contradiction.

(c) If $a_2(z) = c_2 (\alpha - z)(\beta - z), b_2(z) = \frac{1}{c_2} (\alpha - z)z$, then by (\ref{a1b2}), $\beta - z$ is a factor of $a_1(z)$. It follows that $c_1 = 0, a_1(z) = 0$, then we can easily derive a contradiction.

(d) If $a_2(z) = c_2 (\alpha - z)z, b_2(z) = \frac{1}{c_2} (\alpha - z)(\beta - z)$, then by (\ref{a1b2}), $\beta - z$ is a factor of $b_1(z)$, therefore $b_1(z) = c_3 (\overline{\beta}z - 1) (\beta - z)$. Since $\{\rho_1(0), \rho_4(0)\} = \{0, \alpha\}, \{\rho_2(0), \rho_3(0)\} = \{\alpha, \beta\}$, we have
$$\rho_1(z) \rho_4(z) = \frac{a_2(z)}{a_0(z)} = \frac{c_2 z}{1- \overline{\alpha}z} \varphi_\alpha(z),$$
$$\rho_1(z) + \rho_4(z) = - \frac{a_1(z)}{a_0(z)} = - c_1.$$
But by Lemma \ref{psptwpds}, $\rho_1(z) + \rho_4(z)$ is not a constant, we get a contradiction.

Case II: $a_0(z) = 1- \overline{\alpha}z$, $b_0(z) = (1- \overline{\alpha}z) (\overline{\beta}z - 1)$. Then by (\ref{aob1}), $\overline{\beta}z - 1$ is a factor of $b_1(z)$.

(e) If $a_2(z) = c_2 (\beta - z)z, b_2(z) = \frac{1}{c_2} (\alpha - z)^2$, then by (\ref{a1b2}), $(\alpha - z)^2$ is a factor of $b_1(z)$. It follows that $b_1(z) = 0$, then we can easily derive a contradiction.

(f) If $a_2(z) = c_2 (\alpha - z)^2, b_2(z) = \frac{1}{c_2} (\beta - z)z$, this contradicts the fact that $\alpha - z$ is a factor of $a_2(z)$ and $b_2(z)$.

(g) If $a_2(z) = c_2 (\alpha - z)(\beta - z), b_2(z) = \frac{1}{c_2} (\alpha - z)z$, then by (\ref{a1b2}), $\beta - z$ is a factor of $a_1(z)$. Suppose $a_1(z) = (\beta - z) (\alpha_1 z + \beta_1), b_1(z) = (\overline{\beta}z - 1) (\alpha_2 z + \beta_2)$, where $\alpha_i, \beta_i$ are constants. Since $\{\rho_1(0), \rho_4(0)\} = \{0, \alpha\}, \{\rho_2(0), \rho_3(0)\}= \{\alpha, \beta\}$, we have
$$
\begin{cases}
\rho_1(z) \rho_4(z) = \frac{b_2(z)}{b_0(z)} = \frac{1}{c_2}\frac{z}{\overline{\beta}z - 1} \varphi_\alpha(z),\\
\rho_1(z) + \rho_4(z) = - \frac{b_1(z)}{b_0(z)} = \frac{- 1}{1- \overline{\alpha}z} (\alpha_2 z + \beta_2),\\
\rho_2(z) \rho_3(z) = \frac{a_2(z)}{a_0(z)} = c_2 \varphi_\alpha(z) (\beta - z),\\
\rho_2(z) + \rho_3(z) = - \frac{a_1(z)}{a_0(z)} = \frac{- 1}{1- \overline{\alpha}z} (\beta - z) (\alpha_1 z + \beta_1).\\
\end{cases}
$$
Thus by Lemma \ref{psptwpds}, we get
\begin{align}
\rho_{1}(0)+\rho_{4}(0) &  =\alpha=-\beta_{2},\nonumber\\
\rho_{1}(\alpha)+\rho_{4}(\alpha) &  =\beta=\frac{-1}{1-|\alpha|^{2}}%
(\alpha_{2}\alpha+\beta_{2}),\label{a2b2}\\
\rho_{2}(0)\rho_{3}(0) &  =\alpha\beta=c_{2}\alpha\beta,\quad\rho_{2}%
(0)+\rho_{3}(0)=\alpha+\beta=-\beta\beta_{1},\nonumber
\end{align}
Therefore
\[
\beta_{2}=-\alpha,\quad c_{2}=1,\quad\beta_{1}=-\frac{\alpha+\beta}{\beta},
\]
Suppose $\alpha=t\beta$, then $\beta_{1}=-1-t$.

By (\ref{aob1}): $a_{0}(z)b_{1}(z)+a_{1}(z)b_{0}(z)=(1-\overline{\alpha}%
z)^{2}(1-\overline{\beta}z)[(2\alpha+\beta)-z]$, we have
\[
(\alpha_{2}z+\beta_{2})+(\beta-z)(\alpha_{1}z+\beta_{1})=(1-\overline{\alpha
}z)[z-(2\alpha+\beta)].
\]
Comparing the coefficients of $z^{j},j=1,2$, we get
\begin{equation}
\alpha_{1}=\overline{\alpha},\quad\alpha_{2}=2|\alpha|^{2}+1+\beta
_{1}=2|\alpha|^{2}-t\label{a1a21b}%
\end{equation}
Plugging $\beta_{2}=-\alpha$ and $\alpha_{2}=2|\alpha|^{2}-t$ into
(\ref{a2b2}), we have%
\[
-\left(  1-|\alpha|^{2}\right)  \beta=\left(  2|\alpha|^{2}-t\right)
\alpha-\alpha=\left(  2|\alpha|^{2}-1-t\right)  \alpha.
\]
Equivalently for $t=\alpha/\beta$
\begin{equation}
-\left(  1-|\alpha|^{2}\right)  =\left(  2|\alpha|^{2}-1-t\right)
t.\label{t1}%
\end{equation}
In particular, $t$ is real.

By (\ref{a1b2}): $a_{1}(z)b_{2}(z)+a_{2}(z)b_{1}(z)=(\alpha-z)^{2}%
(\beta-z)[1-(2\overline{\alpha}+\overline{\beta})z]$, we obtain
\[
(\alpha_{1}z+\beta_{1})\frac{z}{c_{2}}+c_{2}(\overline{\beta}z-1)(\alpha
_{2}z+\beta_{2})=(\alpha-z)[1-(2\overline{\alpha}+\overline{\beta})z].
\]
Comparing the coefficients of $z^{2},$ we get
\begin{equation}
\frac{\alpha_{1}}{c_{2}}+c_{2}\alpha_{2}\overline{\beta}=2\overline{\alpha
}+\overline{\beta};\label{rba1c2a2b}%
\end{equation}
Recall that $t=\overline{t}=\overline{\alpha}/\overline{\beta},$ $c_{2}=1$,
$\alpha_{1}=\overline{\alpha}$, $\alpha_{2}=2|\alpha|^{2}-t,$ the above
equation becomes
\begin{equation}
\overline{\alpha}+(2|\alpha|^{2}-t)\overline{\beta} =2\overline{\alpha
}+\overline{\beta}\text{ }\ \
\text{or }\ \ (2|\alpha|^{2}-t)   =t+1\label{t2}.%
\end{equation}
Thus $2|\alpha|^{2} - 1 = 2t$. It then follows from (\ref{t1}) that $- (1- |\alpha|^2) = t^2$ - this is a contradiction. This contradiction shows that $a_2(z) = c_2 (\alpha - z)(\beta - z)$ is impossible.

(h) If $a_2(z) = c_2 (\alpha - z)z, b_2(z) = \frac{1}{c_2} (\alpha - z)(\beta - z)$, then by (\ref{a1b2}), $\beta - z$ is a factor of $b_1(z)$. Recall that by (\ref{aob1}), $\overline{\beta}z - 1$ is also a factor of $b_1(z)$, we have $b_1(z) = c_1 (\overline{\beta}z - 1) (\beta - z)$. Since $\{\rho_2(\beta), \rho_3(\beta)\} = \{0, 0\}$, we obtain
$$
\begin{cases}
\rho_2(z) \rho_3(z) = \frac{b_2(z)}{b_0(z)} = \frac{-1}{c_2}\varphi_\alpha(z)\varphi_\beta(z);\\
\rho_2(z) + \rho_3(z) = - \frac{b_1(z)}{b_0(z)} = \frac{- c_1(\beta-z)}{1- \overline{\alpha}z};\\
\rho_1(z) \rho_4(z) = \frac{a_2(z)}{a_0(z)} = c_2 z\varphi_\alpha(z);\\
\rho_1(z) + \rho_4(z) = - \frac{a_1(z)}{a_0(z)}.
\end{cases}
$$
By (\ref{aob1}): $a_0(z) b_1(z) + a_1(z) b_0(z) = (1- \overline{\alpha}z)^2 (1- \overline{\beta}z) [(2\alpha + \beta) -z]$, we get
$$a_1(z) =  [z - (2\alpha + \beta)] (1 - \overline{\alpha}z) - c_1(\beta-z),$$
Thus $\rho_1(z) + \rho_4(z) = - \frac{a_1(z)}{a_0(z)} = (2 \alpha + \beta - z) + \frac{c_1 (\beta - z)}{1- \overline{\alpha}z}$. Then by Lemma \ref{psptwpds}, we have
\begin{align}
\rho_{2}(0)\rho_{3}(0) &  =\alpha\beta=\frac{-\alpha\beta}{c_{2}},\quad
\rho_{2}(0)+\rho_{3}(0)=\alpha+\beta=-c_{1}\beta,\label{eqfdc2c1}\\
\rho_{2}(\alpha)+\rho_{3}(\alpha) &  =\alpha=\frac{-c_{1}(\beta-\alpha
)}{1-|\alpha|^{2}},\label{eqfdc2c22}\\
\rho_{1}(\beta)\rho_{4}(\beta) &  =\alpha^{2}=c_{2}\beta\varphi_{\alpha}%
(\beta),\quad\rho_{1}(\beta)+\rho_{4}(\beta)=2\alpha.\label{deqfac2ba}%
\end{align}
Suppose $t=\alpha/\beta$, then by (\ref{eqfdc2c1}),
\[
c_{2}=-1,\quad c_{1}=-\frac{\alpha+\beta}{\beta}=-1-t.
\]
Now equation (\ref{eqfdc2c22}) becomes%
\begin{align*}
\alpha\left(  1-|\alpha|^{2}\right)   & =(1+t)(\beta-\alpha),\\
\text{or }t\left(  1-|\alpha|^{2}\right)   & =(1+t)(1-t).
\end{align*}
In particular $t$ is real. By (\ref{deqfac2ba}) and $c_{2}=-1$, we get
\[
\varphi_{\alpha}(\beta)=-\frac{\alpha^{2}}{\beta}=-t\alpha.
\]
Recall that
$$
\begin{cases}
p_1(w,z) = a_0(z) w^2 + a_1(z) w + a_2(z);\\
p_2(w,z) = b_0(z) w^2 + b_1(z) w + b_2(z).
\end{cases}
$$
Then
$$
\begin{cases}
p_1(w,z) = (1 - \overline{\alpha}z) w^2 + \Big([z - (2\alpha + \beta)] (1 - \overline{\alpha}z) - c_1(\beta-z)\Big) w - z (\alpha - z);\\
p_2(w,z) = (1 - \overline{\alpha}z)(\overline{\beta}z - 1) w^2 + c_1 (\overline{\beta}z - 1) (\beta - z)w - (\alpha - z) (\beta - z),
\end{cases}
$$
where $c_1 = - 1 -t = -1 - \alpha/\beta$. If $t = \frac{\alpha}{\beta}$ is real, and $\varphi_\beta(\alpha) = t\alpha$, it is then straightforward to verify that
$$f(w,z) = (w - z) p_1(w,z) p_2(w,z).$$

Conversely, if $t = \frac{\alpha}{\beta}$ is real and $\varphi_\beta(\alpha) = t\alpha$, then we have shown that there are irreducible polynomials $p_1(w,z)$ and $p_2(w,z)$ such that
$$
f(w,z) = (w - z) p_1(w,z) p_2(w,z),
$$
and hence the partition for $\phi$ is $\{\{0\}, \{1, 4\}, \{2, 3\}\}$.
\end{proof}

By Propositions \ref{pezfpa}, \ref{ptwpztab}, \ref{ptfpezsabg} and \ref{ptwpstrs}, we obtain the following theorem.
\begin{theorem}\label{urstdpopf}
Let $\phi$ be a finite Blaschke product of order $5$. Then one of the following holds:
\begin{itemize}

\item[(a)] If $\phi$ is equivalent to $\varphi_\alpha^5$ for $\alpha \in \D$, then the partition is $\{\{0\}, \{1\}, \{2\}, \{3\}, \{4\}\}$;
\item[(b)] If $\phi$ is equivalent to $(z^2\varphi_\alpha^2\varphi_\beta) \circ \varphi_\gamma$ for $\alpha, \beta \in \D \backslash\{0\}, \gamma \in \D, \alpha/\beta \in \R, \varphi_\beta(\alpha) = \frac{\alpha^2}{\beta}$, then the partition is $\{\{0\}, \{1, 4\}, \{2, 3\}\}$;
\item[(c)] If $\phi$ is not equivalent to any of the functions in {\rm(a)} and {\rm(b)}, then the partition is $\{\{0\}, \{1, 2, 3, 4\}\}$.
\end{itemize}
\end{theorem}

Remarkably, the hypothesis of Theorem \ref{urstdpopf} (a) is natural. In fact, according to the definiton of the equivalence of two finite Blaschke products, if $\alpha\neq 0$, then
$$\phi_\alpha^{-1}(\phi_\alpha^5) \neq a z^5\ \ \&\ \  \phi_\alpha^{-1}(z^5) \neq a \phi_\alpha^5,
$$
and hence $\phi_\alpha^5$ is not equivalent to $z^5$. Meanwhile, it is worth mentioning here that both $M_{\phi_\alpha^5}$ and $M_{z^5}$ on $B^2$ are unitarily equivalent and so they have the same reducibility on $B^2$ - but on $D^2$ both $M_{\phi_\alpha^5}$ and $M_{z^5}$ are not unitarily equivalent when $\alpha \neq 0$; see Corollary \ref{irddpan}.

On the other hand, when the order $n$ of $\phi$ is prime, it is unknown whether the number $q$ of the connected components of the Riemann surface $S_\phi$ can be different from 2 and $n$ (\cite{DPW12}), and so Theorem \ref{urstdpopf} answers this question when $n$ is $5$. Accordingly, we think that $q$ can also be different from $2$ and $n$ (a prime greater than $5$).

\section{Reducible $M_\phi$ on $D^2$ with $\phi$ being of order $5$}\label{s3}

Now we are ready to discuss the reducibility of $M_\phi$ on $D^2$ when $\phi$ is a finite Blaschke product of order $5$.

\begin{theorem}\label{rdopo5ubp1}
Let $\phi$ be a finite Blaschke product of order $5$. Then $M_\phi$ is reducible on $D^2$ if and only if $\phi$ is equivalent to $z^5$.
\end{theorem}
\begin{proof}
Recall that the possible partitions are
$$
\{\{0\},
\{1\}, \{2\}, \{3\}, \{4\}\};\ \{\{0\}, \{1, 2, 3, 4\}\};\ \{\{0\}, \{1, 4\}, \{2, 3\}\}.
$$
By Theorem \ref{urstdpopf}, these three partitions can happen. Recall also that if the partition is $\{\{0\}, \{1, 2, 3, 4\}\}$, then $M_\phi$ has exact two minimal reducing subspaces on $B^2$: $M_0(\phi); M_0(\phi)^\perp$, where $M_0(\phi) = \text{span}\{\phi'\phi^j: j \geq 0\}$. It follows from Theorem 2.5 \cite{Luo} that $M_\phi$ is irreducible on $D^2$ in this case. Since equivalent Blaschke products have the same reducing subspaces, we can always assume that $\phi(0) = 0$. We have the following cases.

(a) $\phi = z^5$. Then $M_\phi$ has exact $5$ minimal reducing subspaces on $D^2$: $N_j = \text{span} \{z^l: l \equiv j ~\text{mod}~ 5\}, j = 0, 1, \cdots, 4$ (\cite{SZ02}).

(b) $\phi = z^4 \varphi_\alpha, \alpha \neq 0$. By Proposition \ref{pezfpa}, the partition is $\{\{0\}, \{1, 2, 3, 4\}\}$, thus $M_\phi$ is irreducible on $D^2$ in this case.

(c) $\phi = z^3 \varphi_\alpha \varphi_\beta, \alpha\beta \neq 0$. By Proposition \ref{ptwpztab}, the partition is $\{\{0\}, \{1, 2, 3, 4\}\}$, hence $M_\phi$ is irreducible on $D^2$ in this case.

(d) $\phi = z^2 \varphi_\alpha^3, \alpha \neq 0$. Note that $M_{z^{2}\varphi_{\alpha}^{3}}$ is unitarily equivalent to $M_{z^{3}\varphi_{\alpha}^{2}}$. By Proposition \ref{ptwpztab}, $M_{z^{3}\varphi_{\alpha}^{2}}$ has exact two minimal reducing subspaces on $B^{2}$, so $M_{z^{2}\varphi_{\alpha}^{3}}$ has exact two minimal reducing subspaces on $B^{2}$. Therefore by Theorem 2.5 \cite{Luo}, $M_{z^{2}\varphi_{\alpha}^{3}}$ is irreducible on $D^{2}$.

(e) $\phi = z^2 \varphi_\alpha \varphi_\beta \varphi_\gamma, \alpha, \beta, \gamma$ are mutually distinct and $\alpha\beta\gamma \neq 0$. By Proposition \ref{ptfpezsabg}, the partition is $\{\{0\}, \{1, 2, 3, 4\}\}$, so $M_\phi$ is irreducible on $D^2$ in this case.

(f) $\phi = z^2 \varphi_\alpha^2 \varphi_\beta, \alpha\beta \neq 0$. If the partition is $\{\{0\}, \{1, 2, 3, 4\}\}$, then $M_\phi$ is irreducible on $D^2$ in this case. If the partition is $\{\{0\}, \{1, 4\}, \{2, 3\}\}$, we show that $\dim \widetilde{\HA_\phi} = 1$. Suppose $\dim \widetilde{\HA_\phi} \neq 1$. Recall that
$$\HL = \text{span}\left\{(a_1, \cdots, a_q): f \in D^2, \sum_{i=1}^q a_i F_i(0) = 0, \sum_{i=1}^q \overline{a_i} H_i(0) = 0\right\},$$
where
$$
\begin{cases}
F_i(z) = \sum\limits_{\rho \in G_i} f(\rho(z)) \rho(z);\\
H_i(z) = \sum\limits_{\rho \in G_i^{-1}} f(\rho(z)) \rho(z).
\end{cases}
$$
By Proposition \ref{vaiddatlms}, $\dim \HL \neq 1$. Since $\rho_0(0) = 0$, it follows that for $f \in D^2$, the equation
\begin{equation}\label{rcdidugovb}
\sum_{i=1}^3 a_i F_i(0)= a_2 [f(\rho_1(0)) \rho_1(0) + f(\rho_4(0))\rho_4(0)]+ a_3 [f(\rho_2(0)) \rho_2(0) + f(\rho_3(0))\rho_3(0)]= 0
\end{equation}
has a nontrivial solution $(a_2, a_3)$. By Lemma \ref{psptwpds}, we have
$$
\begin{cases}
\{\rho_1(0), \rho_4(0)\} = \{0, \alpha\};\\
\{\rho_2(0), \rho_3(0)\} = \{\alpha, \beta\},
\end{cases}
$$
whence finding that equation (\ref{rcdidugovb}) becomes
$$a_2[f(\alpha)\alpha] + a_3[f(\alpha)\alpha + f(\beta)\beta] = 0.$$
Upon choosing $f(z) = \varphi_\alpha(z)$, we get
$a_3 \varphi_\alpha(\beta)\beta = 0$, whence reaching $a_3 = 0$, $a_2 = 0$, and $(a_2, a_3) = (0, 0)$ - this is a contradiction. Therefore $\dim \widetilde{\HA_\phi} = 1$. In either case, we have $\dim \widetilde{\HA_\phi} = 1$, and $M_\phi$ is irreducible on $D^2$.

(g) $\phi = z\varphi_\alpha\varphi_\beta\varphi_\gamma\varphi_\delta,\ \ \alpha\beta\gamma\delta \neq 0$.

If the partition is $\{\{0\}, \{1\}, \{2\}, \{3\}, \{4\}\}$, then each $\rho_i$ is equivalent to itself, thus each $\rho_i$ extends analytic to the unit disc and has modulus 1 on the unit circle, so each $\rho_i$ is a M\"{o}bius transform. Hence $\phi$ is equivalent to $\varphi_b^5, b \in \D \backslash \{0\}$. By Corollary \ref{irddpan}, $M_\phi$ is irreducible on $D^2$.

If the partition is $\{\{0\}, \{1, 2, 3, 4\}\}$, then $M_\phi$ is irreducible on $D^2$ in this case.

If the partition is $\{\{0\}, \{1, 4\}, \{2, 3\}\}$, then we show $\dim \widetilde{\HA_\phi} = 1$. If $\dim \widetilde{\HA_\phi} \neq 1$, then suppose
$$
\begin{cases}
\{\rho_1(0), \rho_4(0)\} = \{\alpha, \beta\};\\
\{\rho_2(0), \rho_3(0)\}= \{\gamma, \delta\},
\end{cases}
$$
as in (f), for $f \in D^2$, the equation
\begin{align}\label{rcpfdugovb}
\sum_{i=1}^3 a_i F_i(0) = a_2 [f(\alpha) \alpha + f(\beta)\beta] + a_3 [f(\gamma) \gamma + f(\delta)\delta]  = 0
\end{align}
has a nontrivial solution $(a_2, a_3)$.

Next, we show that $\{\gamma, \delta\} = \{\alpha, \beta\}$. If $\{\gamma, \delta\} \neq \{\alpha, \beta\}$, then, without loss of generality, suppose $\delta \not\in \{\alpha, \beta\}$. If $\gamma \in \{\alpha, \beta\}$, choosing $f = \varphi_\alpha \varphi_\beta$ in (\ref{rcpfdugovb}) we have $a_3 f(\delta) \delta = 0$, whence getting $a_3 = 0$ and then $a_2 = 0$ - this is a contradiction. Thus $\gamma \not\in \{\alpha, \beta\}$. If $\gamma \neq \delta$, choosing $f = \varphi_\alpha \varphi_\beta \varphi_\gamma$ in (\ref{rcpfdugovb}) we have $a_3 f(\delta) \delta = 0$, whence finding $a_3 = 0$ and then $a_2 = 0$ - this is also a contradiction. If $\gamma = \delta$, then choosing $f = \varphi_\alpha \varphi_\beta$ in (\ref{rcpfdugovb}) gives $a_3 2f(\delta) \delta = 0$. So $a_3 = 0$, and consequently $a_2 = 0$ - again we have a contradiction. Thus, we have
$$\{\gamma, \delta\} = \{\alpha, \beta\}\ \ \&\ \
\{\rho_1(0), \rho_4(0)\} = \{\alpha, \beta\} = \{\rho_2(0), \rho_3(0)\}.
$$
If $\alpha = \beta$, then $\phi = z \varphi_\alpha^4$, and hence by the remark after Proposition \ref{pezfpa}, $M_\phi$ is irreducible on $D^2$. Thus, we only need to consider $\alpha \neq \beta$.

Let
$$
\begin{cases}
f_1(z) = \rho_1(z) \rho_4(z);\\
g(z) = f_1(\rho_1(z)) + f_1(\rho_4(z)) = z \rho_2(z) + z \rho_3(z).
\end{cases}
$$
Then
$$g(0) = f_1(\alpha) + f_1(\beta) = 0.$$
If
$$
g_1(z) = f_1(\rho_1(z)) f_1(\rho_4(z)) = z^2 \rho_2(z) \rho_3(z),
$$
then
$$g_1(0) = f_1(\alpha) f_1(\beta) = 0,
$$
and hence
$$
f_1(\alpha) = f_1(\beta) = 0.
$$
Consequently,
$$\{\rho_1(\alpha), \rho_4(\alpha)\} = \{0, \alpha\}\ \hbox{or}\ \{0, \beta\}.
$$

If
$$
\{\rho_1(\alpha), \rho_4(\alpha)\} = \{0, \alpha\},
$$
then
$$
\{\rho_2(\alpha), \rho_3(\alpha)\} = \{\beta, \beta\},
$$
and hence
$$g(\alpha) = 2\alpha\beta = f_1(0) + f_1(\alpha).$$
Since $f_1(\alpha) = 0, f_1(0) = \alpha\beta$, we obtain $\alpha\beta = 0$, whence reaching a contradiction.

If
$$\{\rho_1(\alpha), \rho_4(\alpha)\}= \{0, \beta\},
$$
then
$$\{\rho_2(\alpha), \rho_3(\alpha)\} = \{\alpha, \beta\},
$$
and hence
$$g(\alpha) = \alpha^2 + \alpha\beta = f_1(0) + f_1(\alpha).$$
Since
$$f_1(\alpha) = 0, f_1(0) = \alpha\beta,
$$
we get $\alpha = 0$ - this is also a contradiction. Therefore, $\dim \widetilde{\HA_\phi} = 1$. In either case, we have $\dim \widetilde{\HA_\phi} = 1$, so $M_\phi$ is irreducible on $D^2$. The proof is complete.
\end{proof}

\section{Reducible $M_\phi$ on $D^2$ with $\phi$ being of order $6$}\label{s4}

In this section, we discuss the reducing subspaces of $M_\phi$ on $D^2$ when the order of $\phi$ is $6$. To do so, we say that a finite Blaschke product $\phi$ of order $n$ is reducible provided that there are two nontrivial finite Blaschke products $\psi_1$ and $\psi_2$ such that $\phi = \psi_1 \circ \psi_2$, and then need the following lemma.

\begin{lemma}[\cite{DPW12}]\label{dpwrarbp}
For a finite Blaschke product $\phi$ of order $n$, $\phi$ is reducible if and only if $G_{k_1} \cup \cdots \cup G_{k_m}$ forms a nontrivial subgroup of $\Z_n$, for some subset $G_{k_1}, \cdots, G_{k_m}$ of the partition arising from $\phi$.
\end{lemma}

From this lemma we see that if $\phi$ is a finite Blaschke product of order $6$, then the partition $\{\{0\}, \{1, 5\}, \{2, 4\}, \{3\}\}$ is impossible - in fact - if the partition is $\{\{0\}, \{1, 5\}, \{2, 4\}, \{3\}\}$, then $\{0\} \cup \{3\}, \{0\} \cup \{2, 4\}$ are two subgroups of $\Z_6$, and hence the proof of Lemma \ref{dpwrarbp} in \cite{DPW12} produces the finite Blaschke products $\psi_1$ of order $2$, $\psi_2$ of order $3$, $\psi_3$ of order $3$ and $\psi_4$ of order $2$ such that $\phi = \psi_2 \circ \psi_1 = \psi_4 \circ \psi_3$, which is impossible.

Recall that if $\phi$ is a finite Blaschke product of order $n$ then there are $\alpha \in \D$ and a Blaschke product $\psi$ of order $n-2$ such that $\phi$ is equivalent to $\varphi_\alpha^2 \psi$. If $\phi$ is a finite Blaschke product of order $6$ and $\phi$ is reducible, then there are a finite Blaschke product $\psi_1$ of order $2$ and a finite Blaschke product $\psi_2$ of order $3$ such that $\phi = \psi_1 \circ \psi_2$ or $\phi = \psi_2\circ \psi_1$. We have the following cases.

(A) $\phi = \varphi_\alpha^6, \alpha \in \D$. Then the partition for $\phi$ is $\{\{0\}, \{1\}, \{2\}, \{3\}, \{4\}, \{5\}\}$.

(B) $\phi = \psi_1 \circ \psi_2 = \varphi_\alpha^2(\varphi_\beta^3), \alpha , \beta \in \D, \alpha \neq 0$. Then by a little calculation, we see that the Riemann surface $S_\phi$ has four connected components. Since $\rho_0, \rho_2, \rho_4$ are local inverses for $\psi_2$ and analytic in $\D$, the partition for $\phi$ is $\{\{0\}, \{1, 3, 5\}, \{2\}, \{4\}\}$.

(C) $\phi = \psi_2\circ \psi_1 = \varphi_\alpha^3(\varphi_\gamma^2), \alpha, \gamma \in \D, \alpha \neq 0$. Then the Riemann surface $S_\phi$ has four connected components. Since $\rho_0, \rho_3$ are local inverses for $\psi_1$ and analytic in $\D$. By conditions ($A_1$), ($A_2$) and ($A_3$), we obtain that the possible partitions for $\phi$ are $\{\{0\}, \{1, 5\}, \{2, 4\}, \{3\}\}$ and $\{\{0\}, \{1, 4\}, \{2, 5\}, \{3\}\}$. But, $\{\{0\}, \{1, 5\}, \{2, 4\}, \{3\}\}$ doesn't satisfy Lemma \ref{dpwrarbp}, so the partition for $\phi$ is $\{\{0\}, \{1, 4\}, \{2, 5\}, \{3\}\}$.

(D) $\phi = \psi_2\circ \psi_1 = (\varphi_\alpha^2 \varphi_\beta)\circ(\varphi_\gamma^2), \alpha, \beta, \gamma \in \D, \alpha \neq \beta$. Then the Riemann surface $S_\phi$ has three connected components. Since $\rho_0, \rho_3$ are local inverses for $\psi_1$ and analytic in $\D$, the partition is $\{\{0\}, \{1, 2, 4, 5\}, \{3\}\}$.

(E) $\phi = \psi_1 \circ \psi_2 = \varphi_\alpha^2(\varphi_\beta^2\varphi_\gamma), \beta \neq \gamma$. Then the Riemann surface $S_\phi$ has three connected components. Since $\rho_0(z) = z$, $\rho_2(z), \rho_4(z)$ are solutions for $\frac{\psi_2(w) - \psi_2(z)}{w - z} = 0$, it follows that $\rho_2(z)$ and $\rho_4(z)$ are not analytic in $\D$, so the partition is $\{\{0\}, \{1, 3, 5\}, \{2, 4\}\}$.

(F) $\phi$ is not reducible. If $\phi$ is not reducible, suppose $G_1 = \{0\}$, then by Lemma \ref{dpwrarbp} and condition ($A_3$), for any $i \in \{2, \cdots, q\}$, $G_i \neq \{3\}, \{2, 4\}, \{2\}, \{4\}$. Then it is clear that for any $i \in \{2, \cdots, q\}$, $G_i$ is not a singleton. Without loss of generality, suppose $3 \in G_2$. If $q = 3$, then $\#G_3 = 2$, it then follows from ($A_2$) that $G_3^{-1} = G_3$, hence $G_3 = \{2, 4\}$, this is a contradiction. Thus $q = 2$, and $G_2 = \{1, 2, 3, 4, 5\}$. So the partition is $\{\{0\}, \{1, 2, 3, 4, 5\}\}$.

From the above discussion, when $n=6,$ we have the following possible partitions:%
$$
\begin{cases}
\{\{0\},\{1\},\{2\},\{3\},\{4\},\{5\}\};\ \\
\{\{0\},\{1,3,5\},\{2\},\{4\}\};\{\{0\},\{1,4\},\{2,5\},\{3\}\};\\
\{\{0\},\{1,2,4,5\},\{3\}\};\{\{0\},\{1,3,5\},\{2,4\}\};\\
\{\{0\},\{1,2,3,4,5\}\}.
\end{cases}
$$
Furthermore all these partitions actually happen. The above discussion gives an analogue of Theorem \ref{urstdpopf} for $n=6.$

Now we can prove the main result of this section.
\begin{theorem}\label{rdcwohsut}
Let $\phi$ be a finite Blaschke product of order $6$. Then one of the following holds.
\begin{itemize}
\item[(A)] If $\phi$ is equivalent to $z^6$, then $\dim \widetilde{\HA_\phi} = 6$.
\item[(B)] If $\phi$ is equivalent to $\varphi_\alpha^2(z^3), \alpha \neq 0$, then $\dim \widetilde{\HA_\phi} = 3$.
\item[(C)] If $\phi$ is equivalent to $\varphi_\alpha^2 (\varphi_\beta^3(z)), \alpha = \beta^3 \neq 0$, then $\dim \widetilde{\HA_\phi} = 2$.
\item[(D)] If $\phi$ is equivalent to $\varphi_\alpha^3(z^2), \alpha \neq 0$, then $\dim \widetilde{\HA_\phi} = 2$.
\item[(E)] If $\phi$ is equivalent to $\varphi_\alpha^3 (\varphi_\beta^2(z)), \alpha = \beta^2 \neq 0$, then $\dim \widetilde{\HA_\phi} = 3$.
\item[(F)] If $\phi$ is equivalent to $(\varphi_\alpha^2 \varphi_\beta)\circ (z^2), \alpha \neq \beta$, then $\dim \widetilde{\HA_\phi} = 2$.
\item[(G)] If $\phi$ is equivalent to $\varphi_\alpha^2(\varphi_\beta^2\varphi_\gamma), \beta \neq \gamma, \alpha = \beta^2\gamma$, then $\dim \widetilde{\HA_\phi} = 2$.
\item[(H)] If $\phi$ is not equivalent to any of the functions in (A)-(G), then $M_\phi$ is irreducible on $D^2$, i.e. $\dim \widetilde{\HA_\phi} = 1$.
\end{itemize}
Consequently, $M_\phi$ is reducible on $D^2$ if and only if $\phi$ is equivalent to one of the functions in (A)-(G).
\end{theorem}

\begin{proof}
We use the above partitions and Proposition \ref{vaiddatlms} to discuss the reducibility of $M_\phi$ on $D^2$.

(A) $\phi$ is equivalent to $\varphi_\alpha^6$. If $\alpha \neq 0$, then by Corollary \ref{irddpan}, $M_\phi$ is irreducible on $D^2$. If $\alpha = 0$, then $M_\phi$ has exact $6$ minimal reducing subspaces on $D^2$ (\cite{SZ02}):
$$
N_j = \text{span} \{z^l: l \equiv j ~\text{mod}~ 6\}, j = 0, 1, \cdots, 5.
$$

(B) $\phi$ is equivalent to $\varphi_\alpha^2(\varphi_\beta^3(z)), \alpha \neq 0$, then the partition is $\{\{0\}, \{2\}, \{4\}, \{1, 3, 5\}\}$. Without loss of generality, suppose $\phi = \varphi_\alpha^2(\varphi_\beta^3), \alpha \neq 0$, then
\begin{align*}
&\phi(w) - \phi(z) \\&= [\varphi_\alpha(\varphi_\beta^3(w)) - \varphi_\alpha(\varphi_\beta^3(z))] [\varphi_\alpha(\varphi_\beta^3(w)) + \varphi_\alpha(\varphi_\beta^3(z))]\\
& = \frac{(|\alpha|^2 - 1) (\varphi_\beta^3(w) - \varphi_\beta^3 (z))}{(1- \overline{\alpha}\varphi_\beta^3(w)) (1- \overline{\alpha}\varphi_\beta^3(z))}[\varphi_\alpha(\varphi_\beta^3(w)) + \varphi_\alpha(\varphi_\beta^3(z))]\\
& = \frac{(|\alpha|^2 - 1) \big(\varphi_\beta(w) - \varphi_\beta (z)\big) \big(\varphi_\beta(w) - \omega_3\varphi_\beta (z)\big) \big(\varphi_\beta(w) - \omega_3^2\varphi_\beta (z)\big) }{(1- \overline{\alpha}\varphi_\beta^3(w)) (1- \overline{\alpha}\varphi_\beta^3(z))[\varphi_\alpha(\varphi_\beta^3(w)) + \varphi_\alpha(\varphi_\beta^3(z))]^{-1}},
\end{align*}
where $\omega_3 = e^{\frac{2\pi i}{3}}$. Thus by Theorem \ref{ccrsrsbd},
$$\rho_0(z) = z, \quad \rho_2(z) = \varphi_\beta(\omega_3\varphi_\beta(z)), \quad \rho_4(z) = \varphi_\beta(\omega_3^2\varphi_\beta(z)),$$
and $\rho_1(z), \rho_3(z), \rho_5(z)$ are solutions for $\varphi_\alpha(\varphi_\beta^3(w)) + \varphi_\alpha(\varphi_\beta^3(z)) = 0$.
It follows that
$$\rho_0(0) = 0, \rho_2(0) = \varphi_\beta(\omega_3 \beta): = \alpha_1, \rho_4(0) = \varphi_\beta(\omega_3^2 \beta): = \beta_1,
$$
and
$$
\rho_1(0): = \gamma_1, \rho_3(0): = \gamma_2, \rho_5(0): = \gamma_3
$$
satisfy
$$
\varphi_\alpha(\varphi_\beta^3(\gamma_i)) = - \varphi_\alpha(\beta^3).
$$

We use Proposition \ref{vaiddatlms} to determine the dimension of $\widetilde{\HA_\phi}$. Recall that for $f \in D^2$,
$$
F_i(z) = \sum_{\rho \in G_i}f(\rho(z)) \rho(z))\ \&\ H_i(z) = \sum_{\rho \in G_i^{-1}}f(\rho(z)) \rho(z)).
$$
Since $G_i^{-1} = G_i$, we have $H_i(z) = F_i(z)$. Suppose $a_1, a_2, a_3, a_4$ satisfy
\begin{align}\label{eqfa246p23}
&\sum_{i=1}^4 a_i F_i(0)\\ & = a_2 f(\rho_2(0)) \rho_2(0) + a_3 f(\rho_4(0)) \rho_4(0) \notag\\
&\hspace{0.5cm}  + a_4 [f(\rho_1(0)) \rho_1(0) + f(\rho_3(0)) \rho_3(0) + f(\rho_5(0)) \rho_5(0)] \notag\\
& = a_2 f(\alpha_1) \alpha_1 + a_3 f(\beta_1)\beta_1 + a_4 [f(\gamma_1) \gamma_1 + f(\gamma_2) \gamma_2 + f(\gamma_3) \gamma_3] \notag\\
& = 0, \notag
\end{align}
and
$$\sum_{i=1}^4 \overline{a_i} H_i(0) = \sum_{i=1}^4 \overline{a_i} F_i(0) = 0.$$

(i) $\beta = 0$. Then $\alpha_1 = \beta_1 = 0$, $\gamma_i$ are solutions for $\varphi_\alpha(-\gamma^3) = - \alpha$, thus $\gamma_1, \gamma_2, \gamma_3$ are mutually distinct and not zero, and (\ref{eqfa246p23}) becomes
$$a_4 [f(\gamma_1) \gamma_1 + f(\gamma_2) \gamma_2 + f(\gamma_3) \gamma_3] = 0.$$
Choosing $f = \varphi_{\gamma_1}\varphi_{\gamma_2}$ we get $a_4 = 0$. It follows that $a_1, a_2, a_3$ are arbitrary, so $\dim \widetilde{\HA_\phi} = \dim \HL = 3$, and $M_\phi$ is reducible on $D^2$.

(ii) $\beta \neq 0, \alpha = \beta^3$. Then $\varphi_\alpha (\varphi_\beta^3(z)) = a z \varphi_{b_1} \varphi_{b_2}, |a| = 1, b_1, b_2 \in \D$. Note that
$$\varphi_\alpha (\varphi_\beta^3(\alpha_1)) = \varphi_\alpha((\omega_3 \beta)^3) = \varphi_\alpha(\beta^3) = 0,$$
and similarly, $\varphi_\alpha (\varphi_\beta^3(\beta_1)) = 0$. So, it follows that $\varphi_\alpha (\varphi_\beta^3(z)) = a z \varphi_{\alpha_1} \varphi_{\beta_1}$. Without loss of generality, suppose $\gamma_1 = 0, \gamma_2 = \alpha_1, \gamma_3 = \beta_1$, then (\ref{eqfa246p23}) becomes
$$a_2 f(\alpha_1) \alpha_1 + a_3 f(\beta_1)\beta_1 + a_4 [f(\alpha_1) \alpha_1 + f(\beta_1) \beta_1] = 0,$$
therefore $a_2 = -a_4 = a_3$, so $\dim \widetilde{\HA_\phi} = \dim \HL = 2$, and $M_\phi$ is reducible on $D^2$.

(iii) $\beta \neq 0, \alpha \neq \beta^3$. Then $\alpha_1, \beta_1, \gamma_1, \gamma_2, \gamma_3$ are mutually distinct and not zero. Upon taking
$$f = \varphi_{\beta_1}\varphi_{\gamma_1}\varphi_{\gamma_2}\varphi_{\gamma_3}\ \ \hbox{or}\ \  \varphi_{\gamma_1}\varphi_{\gamma_2}\varphi_{\gamma_3}, \varphi_{\gamma_2}\varphi_{\gamma_3}
$$
in (\ref{eqfa246p23}), we get $a_2 = a_3 = a_4 = 0$. Thus $\dim \widetilde{\HA_\phi} = \dim \HL = 1$, and $M_\phi$ is irreducible on $D^2$.

(C) $\phi$ is equivalent to $\varphi_\alpha^3(\varphi_\beta^2(z)), \alpha \neq 0$, then the partition is $\{\{0\}, \{3\}, \{1, 4\}, \{2, 5\}\}$. Without loss of generality, suppose $\phi = \varphi_\alpha^3(\varphi_\beta^2), \alpha \neq 0$, by the same calculation in (B), we obtain $\rho_0(z) = z, \rho_3(z) = \varphi_\beta(-\varphi_\beta(z))$, $\rho_1(z), \rho_4(z)$ are solutions for $\varphi_\alpha(\varphi_\beta^2(w)) = \omega_3\varphi_\alpha(\varphi_\beta^2(z))$, and $\rho_2(z), \rho_5(z)$ are solutions for $\varphi_\alpha(\varphi_\beta^2(w)) = \omega_3^2\varphi_\alpha(\varphi_\beta^2(z))$, where $\omega_3 = e^{2\pi i/3}$. It follows that
$$\rho_0(0) = 0,\ \rho_3(0) = \varphi_\beta(-\beta): = \alpha_1,\  \rho_1(0): = \beta_1, \rho_4(0): = \beta_2
$$
satisfy
$$
\varphi_\alpha(\varphi_\beta^2(\beta_i)) = \omega_3\varphi_\alpha(\beta^2),
$$
and
$$
\rho_2(0): = \gamma_1,\  \rho_5(0): = \gamma_2
$$
satisfy
$$
\varphi_\alpha(\varphi_\beta^2(\gamma_i)) = \omega_3^2\varphi_\alpha(\beta^2).
$$

For $f \in D^2$, suppose $a_1, a_2, a_3, a_4$ satisfy
\begin{equation}\label{eqpa32s}
\sum_{i=1}^4 a_i F_i(0)= a_2 f(\alpha_1) \alpha_1 + a_3 [f(\beta_1)\beta_1 + f(\beta_2)\beta_2] + a_4 [f(\gamma_1) \gamma_1 + f(\gamma_2) \gamma_2] = 0,
\end{equation}
and
$$\sum_{i=1}^4 \overline{a_i} F_i(0) = 0.$$

(i) $\beta = 0$. Then $\alpha_1 = 0$, $\beta_1, \beta_2, \gamma_1, \gamma_2$ are mutually distinct and not zero, and (\ref{eqpa32s}) becomes
$$a_3 [f(\beta_1)\beta_1 + f(\beta_2)\beta_2] + a_4 [f(\gamma_1) \gamma_1 + f(\gamma_2) \gamma_2].$$
Choosing $f = \varphi_{\beta_1}\varphi_{\beta_2}\varphi_{\gamma_1}, \varphi_{\beta_1}$, we obtain $a_3 = a_4 = 0$. Thus $\dim \widetilde{\HA_\phi} = \dim \HL = 2$, and $M_\phi$ is reducible on $D^2$.

(ii) $\beta \neq 0, \alpha = \beta^2$. Then $\alpha_1 \neq 0$, and $\beta_1, \beta_2, \gamma_1, \gamma_2$ satisfy
$$\varphi_\alpha(\varphi_\beta^2(z)) = 0.$$
Since $0, \alpha_1$ satisfy $\varphi_\alpha(\varphi_\beta^2(z)) = 0$, it follows that $\{\beta_1, \beta_2\} = \{\gamma_1, \gamma_2\} = \{0, \alpha_1\}$, and (\ref{eqpa32s}) becomes
$$a_2 f(\alpha_1) \alpha_1 + a_3 f(\alpha_1) \alpha_1 + a_4 f(\alpha_1) \alpha_1 = 0.$$
Thus $a_2 + a_3 + a_4 = 0$. It follows that $\dim \widetilde{\HA_\phi} = \dim \HL = 3$, and $M_\phi$ is reducible on $D^2$.

(iii) $\beta \neq 0, \alpha \neq \beta^2$. Then $\alpha_1, \beta_1, \beta_2, \gamma_1 \gamma_2$ are mutually distinct and not zero. By the same argument in (B), we have $\dim \widetilde{\HA_\phi} = \dim \HL = 1$, and $M_\phi$ is irreducible on $D^2$.

(D) If $\phi$ is equivalent to $(\varphi_\alpha^2 \varphi_\beta)\circ(\varphi_\gamma^2), \alpha \neq \beta$, then the partition is $\{\{0\}, \{3\}, \{1, 2, 4, 5\}$. Suppose $(\varphi_\alpha^2 \varphi_\beta)\circ(\varphi_\gamma^2), \alpha \neq \beta$, by the same calculation in (B), we obtain that
$$
\rho_0(z) = z,\ \rho_3(z) = \varphi_\gamma(-\varphi_\gamma(z)),
$$
and $\rho_1(z), \rho_2(z), \rho_4(z), \rho_5(z)$ are solutions of
$$
[\phi(w) - \phi(z)] / [(\varphi_\gamma(w) - \varphi_\gamma(z)) (\varphi_\gamma(w) + \varphi_\gamma(z))] = 0.
$$
It follows that
$$
\rho_0(0) = 0, \rho_3(0) = \varphi_\gamma(-\gamma): = \alpha_1,
$$
and
$$
\rho_1(0): = \beta_1,\ \rho_2(0): = \beta_2, \rho_4(0): = \beta_3, \rho_5(0): = \beta_4
$$
satisfy
$$[\phi(w) - \phi(0)] / [(\varphi_\gamma(w) - \gamma) (\varphi_\gamma(w) + \gamma)] = 0.$$

For $f \in D^2$, suppose $a_1, a_2, a_3$ satisfy
\begin{equation}\label{eqfp232s}
\sum_{i=1}^3 a_i F_i(0) = a_2 f(\alpha_1) \alpha_1 + a_3 [f(\beta_1)\beta_1 + f(\beta_2)\beta_2 + f(\beta_3) \beta_3 + f(\beta_4) \beta_4] = 0,
\end{equation}
and
$$\sum_{i=1}^3 \overline{a_i} F_i(0) = 0.$$

(i) $\gamma = 0$. Then $\alpha_1 = 0$, but not all $\beta_i$ are zero, it follows that $a_3 = 0$, so $\dim \widetilde{\HA_\phi} = \dim \HL = 2$, and $M_\phi$ is reducible on $D^2$.

(ii) $\gamma \neq 0$. We claim that there exists $\beta_i$ such that $\beta_i \neq 0$ and $\beta_i \neq \alpha_1$. If we assume this claim, then by the same argument in (B)-(iii), we obtain that $a_2 = a_3 = 0$, thus $\dim \widetilde{\HA_\phi} = \dim \HL = 1$, and $M_\phi$ is irreducible on $D^2$.

Now we prove the claim. We prove it by contradiction. Let
$$
\begin{cases}
f_1(z) = \rho_1(z) + \rho_2(z) + \rho_4(z) + \rho_5(z);\\
g(z) = f_1(\rho_1(z)) + f_1(\rho_2(z)) + f_1(\rho_4(z)) + f_1(\rho_5(z)).
\end{cases}
$$
Then $g(z) = 2 f_1(z) + 4 \rho_3(z) + 4 z$. There are five cases.

If $\{\beta_1, \beta_2, \beta_3, \beta_4\} = \{0, 0, 0, 0\}$, then
$$g(0) = 2f_1(0) + 4\alpha_1 + 0 = 4f_1(0) = 0,$$
so $\alpha_1 = 0$, this is a contradiction.

If $\{\beta_1, \beta_2, \beta_3, \beta_4\} = \{0, 0, 0, \alpha_1\}$, then
$$g(0) = 2f_1(0) + 4\alpha_1 = 3f_1(0) + f_1(\alpha_1).$$
Note that $f_1(0) = \alpha_1$, $f_1(\alpha_1) = 0$ or $\alpha_1$, we get $\alpha_1 = 0$, this is a contradiction.

If $\{\beta_1, \beta_2, \beta_3, \beta_4\} = \{0, 0, \alpha_1, \alpha_1\}$, then $$\phi^{-1}\circ \phi(\alpha_1) = \phi^{-1}\circ \phi(0) = \{0, 0, 0, \alpha_1, \alpha_1, \alpha_1\}.
$$ We show that $\alpha \neq \gamma^2$. If $\alpha = \gamma^2$, then $$
\begin{cases}
\varphi_\alpha(\varphi_\gamma^2(z)) = a_1 z\varphi_{\alpha_1},\ |a_1| = 1;\\
\varphi_\beta(\varphi_\gamma^2(z)) = a_2 \varphi_{b_1}\varphi_{b_2},\ |a_2| = 1,\ b_1 \neq b_2;\\
\alpha_1\notin\{b_1,b_2\},\ b_1b_2 \neq 0.
\end{cases}
$$
Hence
$$\phi(z) = a_1^2 a_2 z^2 \varphi_{\alpha_1}^2 \varphi_{b_1}\varphi_{b_2}.
$$
This implies
$$
\phi^{-1}\circ \phi(0) = \{0, 0, \alpha_1, \alpha_1, b_1, b_2\} \neq \{0, 0, 0, \alpha_1, \alpha_1, \alpha_1\},
$$
which is a contradiction. Therefore $\alpha \neq \gamma^2$.

Note that
$$
\frac{\phi(0) - \phi(z)}{1 - \overline{\phi(0)}\phi(z)} = a z^3 \varphi_{\alpha_1}^3(z), |a| = 1.
$$
So, letting $\phi_1(z) = z^3 \varphi_{\alpha_1}^3$ yields that for $z \in \D$, $\phi'(z) = 0$ if and only if $\phi_1'(z) = 0$. Since
$$\phi'(z) = \varphi_\alpha(\varphi_\gamma^2(z)) [ 2\varphi_\alpha'(\varphi_\gamma^2(z))\varphi_\beta(\varphi_\gamma^2(z)) + \varphi_\alpha(\varphi_\gamma^2(z)) \varphi_\beta(\varphi_\gamma^2(z))] 2\varphi_\gamma(z) \varphi_\gamma'(z),$$
if $z = \gamma$ or $\varphi_\gamma^2(z) = \alpha$, then $\phi'(z) = 0$, but
$$\phi_1'(z) = 3z^2 \varphi_{\alpha_1}^2(z) [\varphi_{\alpha_1} (z)+ z \varphi_{\alpha_1}'(z)]
$$ only has one zero different from $0$ and $\alpha_1$, this is a contradiction. Thus
$$\{\beta_1, \beta_2, \beta_3, \beta_4\} = \{0, 0, \alpha_1, \alpha_1\}$$ is impossible.

If $\{\beta_1, \beta_2, \beta_3, \beta_4\} = \{0, \alpha_1, \alpha_1, \alpha_1\}$, then
$$g(0) = 2f_1(0) + 4\alpha_1 = f_1(0) + 3f_1(\alpha_1).$$
Since $f_1(0) = 3\alpha_1$, we have $f_1(\alpha_1) = \frac{7}{3}\alpha_1$ - but this is impossible.

If $\{\beta_1, \beta_2, \beta_3, \beta_4\} = \{\alpha_1, \alpha_1, \alpha_1, \alpha_1\}$, then
$$g(0) = 2f_1(0) + 4\alpha_1 = 4f_1(\alpha_1).$$
Since $f_1(0) = 4 \alpha_1$, we get $f_1(\alpha_1) = 2 \alpha_1$. Note that
$$\phi^{-1}\circ \phi(\alpha_1) = \phi^{-1}\circ \phi(0) = \{0, \alpha_1, \alpha_1, \alpha_1, \alpha_1, \alpha_1\}\ \&\ \rho_0(\alpha_1) = \alpha_1, \rho_3(\alpha_1) = 0.
$$
Thus $f_1(\alpha_1) = 4 \alpha_1$ - this is a contradiction. Therefore the claim is proved.

(E) $\phi$ is equivalent to $\varphi_\alpha^2 \circ (\varphi_\beta^2 \varphi_\gamma), \beta \neq \gamma$, then the partition is $\{\{0\}, \{2, 4\}, \{1, 3, 5\}$. Suppose
$$\varphi_\alpha^2 \circ (\varphi_\beta^2 \varphi_\gamma), \beta \neq \gamma.
$$
If $\psi(z) = \varphi_\beta^2(z) \varphi_\gamma(z)$, then
$$
\{\rho_0(z) = z,\ \rho_2(z),\ \rho_4(z)\}
$$
are solutions for $\frac{\psi(w) - \psi(z)}{w - z} = 0$, and
$$
\{\rho_1(z), \rho_3(z), \rho_5(z)\}
$$
are solutions of
$$
\varphi_\alpha(\psi(w)) = - \varphi_\alpha(\psi(z)).
$$
Thus
$$\rho_0(0) = 0,\  \rho_2(0): = \alpha_1,\  \rho_4(0): = \alpha_2
$$
satisfy
$$\frac{\psi(\alpha_i) - \psi(0)}{z} = 0,
$$
and
$$\rho_1(0): = \beta_1,\  \rho_3(0): = \beta_2, \rho_5(0): = \beta_3
$$
satisfy
$$\varphi_\alpha(\psi(\beta_j)) = - \varphi_\alpha(\psi(0)).
$$

For $f \in D^2$, suppose $a_1, a_2, a_3$ satisfy
\begin{equation}\label{efa23a2b2g}
\sum_{i=1}^3 a_i F_i(0)\\
= a_2 [f(\alpha_1) \alpha_1 + f(\alpha_2) \alpha_2] + a_3 [f(\beta_1)\beta_1 + f(\beta_2)\beta_2 + f(\beta_3) \beta_3] = 0,
\end{equation}
and
$$\sum_{i=1}^3 \overline{a_i} F_i(0) = 0.$$

(i) $\psi(0) = \alpha$, i.e., $\alpha = \beta^2 \gamma$. Then
$$\varphi_\alpha(\psi(z)) = a z \varphi_{b_1} \varphi_{b_2},\ \  |a| = 1, b_1, b_2 \in \D
$$ and not both $b_1$ and $b_2$ are zero. Observing
$$
\psi(w) = \psi(z)\Leftrightarrow\varphi_\alpha(\psi(w)) = \varphi_\alpha(\psi(z)),
$$
we have
$$
\{\rho_2(0), \rho_4(0)\} = \{b_1, b_2\}\ \ \&\ \  \{\rho_1(0), \rho_3(0), \rho_5(0)\} = \{0, b_1, b_2\}.
$$
So (\ref{efa23a2b2g}) becomes
$$a_2 [f(b_1) b_1 + f(b_2) b_2] + a_3 [f(b_1) b_1 + f(b_2) b_2] = 0.$$
Therefore $a_2 = - a_3$, so $\dim \widetilde{\HA_\phi} = \dim \HL = 2$, then $M_\phi$ is reducible on $D^2$.

(ii) $\alpha \neq \beta^2 \gamma$. Note that $\beta \neq \gamma$, we have at least one $\alpha_i$ is not zero. Since
$$
\begin{cases}
\varphi_\alpha(\psi(\alpha_i)) = \varphi_\alpha(\psi(0)) \neq 0;\\
\varphi_\alpha(\psi(\beta_i)) = -\varphi_\alpha(\psi(0)) \neq \varphi_\alpha(\psi(0)),
\end{cases}
$$
 it follows that
 $$
 \begin{cases}
 \beta_1\beta_2\beta_3 \neq 0;\\
 \alpha_i \neq \beta_j\  \ \forall\ (i,j)\in\{1, 2\}\times\{1, 2, 3\}.
 \end{cases}
 $$
 Then by the same argument in (B)-(iii), we get $a_2 = a_3 = 0$, thus $\dim \widetilde{\HA_\phi} = \dim \HL = 1$, so $M_\phi$ is irreducible on $D^2$.

(F) $\phi$ is not reducible, then the partition is $\{\{0\}, \{1, 2, 3, 4, 5\}\}$. Thus $M_\phi$ has exact two minimal reducing subspaces on $B^2$: $M_0(\phi), M_0(\phi)^\perp$. By Theorem 2.5 \cite{Luo}, $M_\phi$ is irreducible on $D^2$. The proof is complete.
\end{proof}

\section{Reducible $M_\phi$ on $D^2$ with $\phi$ being of order $7$}\label{s5}

In this section we show that if $\phi$ is a finite Blaschke product of order $7$, then $M_\phi$ is reducible on $D^2$ if and only if $\phi$ is equivalent to $z^7$. First we discuss the possible partitions $\{G_1, G_2, \cdots, G_q\}$ when the order of $\phi$ is $7$. When $n = 7$, by \cite[Corollary 8.4]{DSZ11}, $q \neq 6$. Without loss of generality, suppose $G_1 = \{0\}$. Let $m = \min\{\#G_2, \cdots, \# G_q\}$. If $m = 1$, then by condition $(A_3)$, we have $q = 7$. Thus we only need to consider $m \geq 2$.

(i)  If $q = 7$, then the partition is $\{\{0\}, \{1\}, \{2\}, \{3\}, \{4\}, \{5\}, \{6\}\}$.

(ii) If $q = 5, m = 2$, then $\#G_2, \#G_3, \#G_4, \#G_5 \geq 2$, which is impossible.

(iii) If $q = 4, m = 2$, then $\#G_2 = \#G_3 = \#G_4 = 2$.

$(a_1)$ If $G_2 = \{1, 2\}$, then by condition $(A_2)$, $G_3 = \{5, 6\}$, and hence $G_4 = \{3, 4\}$. But this case doesn't satisfy condition $(A_3)$;

$(a_2)$ If $G_2 = \{1, 3\}$, then by condition $(A_2)$, $G_3 = \{4, 6\}$, and hence $G_4 = \{2, 5\}$. But this case doesn't satisfy condition $(A_3)$;

$(a_3)$ If $G_2 = \{1, 4\}$, then by condition $(A_2)$, $G_3 = \{3, 6\}$, and hence $G_4 = \{2, 5\}$. But this case doesn't satisfy condition $(A_3)$;

$(a_4)$ If $G_2 = \{1, 5\}$, then by condition $(A_2)$, $G_3 = \{2, 6\}$, and hence $G_4 = \{3, 4\}$. But this case doesn't satisfy condition $(A_3)$;

$(a_5)$ If $G_2 = \{1, 6\}$, then $G_2 + G_2 = \{2, 0, 0, 5\}$, and hence  condition $(A_3)$ yields $G_3 = \{2, 5\}$ and $G_4 = \{3, 4\}$. This case also satisfies condition $(A_4)$;

$(a_6)$ If $G_2 = \{2, 3\}$, then condition $(A_2)$ implies $G_3 = \{4, 5\}$, and hence $G_4 = \{1, 6\}$. But this case doesn't satisfy condition $(A_3)$;

$(a_7)$ If $G_2 = \{2, 4\}$, then condition $(A_2)$ gives $G_3 = \{3, 5\}$, and hence $G_4 = \{1, 6\}$. But this case doesn't satisfy condition $(A_3)$;

$(a_8)$ If $G_2 = \{2, 5\}$, then $G_2 + G_2 = \{4, 0, 0, 3\}$, and hence condition $(A_3)$ yields $G_3 = \{3, 4\}$ and $G_4 = \{1, 6\}$. This case is essentially the case $(a_5)$;

$(a_9)$ If $G_2 = \{2, 6\}$, then condition $(A_2)$ yields $G_3 = \{1, 5\}$, and hence $G_4 = \{3, 4\}$. But this case doesn't satisfy condition $(A_3)$;

$(a_{10})$ If $G_2 = \{3, 4\}$, then $G_2 + G_2 = \{6, 0, 0, 1\}$, and hence condition $(A_3)$ implies $G_3 = \{1, 6\}$ and $G_4 = \{2, 5\}$. This case is essentially the case $(a_5)$;

$(a_{11})$ If $G_2 = \{3, 5\}$, then by condition $(A_2)$ one has $G_3 = \{2, 4\}$ and so $G_4 = \{1, 6\}$. But this case doesn't satisfy condition $(A_3)$;

$(a_{12})$ If $G_2 = \{3, 6\}$, then by condition $(A_2)$ one has $G_3 = \{1, 4\}$ and so $G_4 = \{2, 5\}$. But this case doesn't satisfy condition $(A_3)$;

$(a_{13})$ If $G_2 = \{4, 5\}$, then by condition $(A_2)$ one has $G_3 = \{2, 3\}$ and so $G_4 = \{1, 6\}$. But this case doesn't satisfy condition $(A_3)$;

$(a_{14})$ If $G_2 = \{4, 6\}$, then by condition $(A_2)$ one has $G_3 = \{1, 3\}$ and so $G_4 = \{2, 5\}$. But this case doesn't satisfy condition $(A_3)$;

$(a_{15})$ If $G_2 = \{5, 6\}$, then by condition $(A_2)$ one has $G_3 = \{1, 2\}$ and so $G_4 = \{3, 4\}$. But this case doesn't satisfy condition $(A_3)$.

Therefore if $q = 4, m = 2$, then we have the possible partition $\{\{0\}, \{1, 6\}, \{2, 5\}, \{3, 4\}\}$.

(iv) If $q = 3, m = 2$, then $\#G_2 = 2, \#G_3 = 4$. By condition $(A_2)$ one has $G_2 = G_2^{-1}$.

$(b_1)$ $G_2 = \{1, 6\}, G_3 = \{2, 3, 4, 5\}$;

$(b_2)$ $G_2 = \{2, 5\}, G_3 = \{1, 3, 5, 6\}$;

$(b_3)$ $G_2 = \{3, 4\}, G_3 = \{1, 2, 5, 6\}$.

Cases $(b_1), (b_2)$ and $(b_3)$ don't satisfy condition ($A_3$).

(v) If $q = 3, m = 3$, then $\#G_2 = \#G_3 = 3$.

$(c_1)$ $G_2 = \{1, 2, 3\}, G_3 = \{4, 5, 6\}$;

$(c_2)$ $G_2 = \{1, 2, 4\}, G_3 = \{3, 5, 6\}$;

$(c_3)$ $G_2 = \{1, 2, 5\}, G_3 = \{3, 4, 6\}$;

$(c_4)$ $G_2 = \{1, 2, 6\}, G_3 = \{3, 4, 5\}$;

$(c_5)$ $G_2 = \{1, 3, 4\}, G_3 = \{2, 5, 6\}$;

$(c_6)$ $G_2 = \{1, 3, 5\}, G_3 = \{2, 4, 6\}$;

$(c_7)$ $G_2 = \{1, 3, 6\}, G_3 = \{2, 4, 5\}$;

$(c_8)$ $G_2 = \{1, 4, 5\}, G_3 = \{2, 3, 6\}$;

$(c_9)$ $G_2 = \{1, 4, 6\}, G_3 = \{2, 3, 5\}$;

$(c_{10})$ $G_2 = \{1, 5, 6\}, G_3 = \{2, 3, 4\}$.

The above cases don't satisfy ($A_3$) except the case $(c_2)$ enjoying $(A_4)$. Therefore, if $q=3,m=3$ then we have the only possible partition $\{1,2,4\},\{3,5,6\}$.

(vi) If $q = 2$, then the partition is $\{\{0\}, \{1, 2, 3, 4, 5, 6\}\}$.

From the above discussions, when $n = 7$, we have the following possible partitions:
$$
\begin{cases}
\{\{0\}, \{1\}, \{2\}, \{3\}, \{4\}, \{5\}, \{6\}\};\\
\{\{0\}, \{1, 6\}, \{2, 5\}, \{3, 4\}\};\\
 \{\{0\}, \{1, 2, 4\}, \{3, 5, 6\}\};\\
 \{\{0\}, \{1, 2, 3, 4, 5, 6\}\}.
 \end{cases}
 $$

\begin{theorem}\label{rdop7m2p1}
Let $\phi$ be a finite Blaschke product of order $7$. Then $M_\phi$ is reducible on $D^2$ if and only if $\phi$ is equivalent to $z^7$.
\end{theorem}

\begin{proof}
Note that if the partition is $\{\{0\}, \{1, 2, 3, 4, 5, 6\}\}$, then by Theorem 2.5 \cite{Luo}, we have that $M_\phi$ is irreducible on $D^2$, i.e. $\dim \widetilde{\HA_\phi} = 1$. Since equivalent Blaschke products have the same reducing subspaces, we can always assume that $\phi(0) = 0$. We have the following cases.

(a) $\phi = z^7$. Then $M_\phi$ has exact $7$ minimal reducing subspaces on $D^2$:
$$
N_j = \text{span} \{z^l: l \equiv j ~\text{mod}~7\},\  j = 0, 1, \cdots, 6.
$$

(b) $\phi = z^6 \varphi_\alpha, \alpha \neq 0$.  We show that the partition is $\{\{0\}, \{1, 2, 3, 4, 5, 6\}\}$. Suppose on the contrary that the partition is not $\{\{0\}, \{1, 2, 3, 4, 5, 6\}\}$, then the partition is
$$
\{\{0\}, \{1, 6\}, \{2, 5\}, \{3, 4\}\}\ \ \hbox{or} \ \{\{0\}, \{1, 2, 4\}, \{3, 5, 6\}\}.
$$

(i) If the partition is $\{\{0\}, \{1, 6\}, \{2, 5\}, \{3, 4\}\}$, without loss of generality, suppose
$$
\{\rho_1(0), \rho_6(0)\} = \{0, 0\},\ \{\rho_2(0), \rho_5(0)\} = \{0, 0\},\  \{\rho_3(0), \rho_4(0)\} = \{0, \alpha\}.
$$
If
$$f(z) = \rho_1(z) + \rho_6(z)\quad\&\quad g(z) = f(\rho_2(z)) + f(\rho_5(z)),
$$
then
$$g(z) = f(z) + \rho_3(z) + \rho_4(z),
$$
and hence
$$g(0) = \alpha = 2 f(0) = 0,$$
this is a contradiction.

(ii) If the partition is $\{\{0\}, \{1, 2, 4\}, \{3, 5, 6\}\}$, suppose
$$
\{\rho_1(0), \rho_2(0), \rho_4(0)\} = \{0, 0, 0\},\quad \{\rho_3(0), \rho_5(0), \rho_6(0)\} = \{0, 0, \alpha\}.
$$
If
$$
\begin{cases}
f(z) = \rho_1(z) + \rho_2(z) + \rho_4(z);\\
g(z) = f(\rho_1(z)) + f(\rho_2(z)) + f(\rho_4(z)) = f(z) + 2 (\rho_3(z) + \rho_5(z) + \rho_6(z)),
\end{cases}
$$
then
$$g(0) = 2\alpha = 3f(0) = 0.$$
This is also a contradiction.

Thus the partition is $\{\{0\}, \{1, 2, 3, 4, 5, 6\}\}$, it follows that $M_\phi$ is irreducible on $D^2$.

(c) $\phi = z^5 \varphi_\alpha \varphi_\beta, \alpha \neq \beta, \alpha\beta \neq 0$. We show that $\dim \widetilde{\HA_\phi} = 1$. Suppose on the contrary that $\dim \widetilde{\HA_\phi} \neq 1$, then the partition is:
$$
\{\{0\}, \{1, 6\}, \{2, 5\}, \{3, 4\}\}\  \ \hbox{or}\  \ \{\{0\}, \{1, 2, 4\}, \{3, 5, 6\}\}.
$$

(i) If the partition is $\{\{0\}, \{1, 6\}, \{2, 5\}, \{3, 4\}\}$, then, without loss of generality, suppose $\{\rho_1(0), \rho_6(0)\} = \{0, 0\}$. If
$$
f(z) = \rho_1(z) + \rho_6(z)\ \ \&\ \ g(z) = f(\rho_1(z)) + f(\rho_6(z)) = 2z + \rho_2(z) + \rho_5(z),
$$
then
$$g(0) = \rho_2(0) + \rho_5(0) = 2 f(0) = 0.$$

If $\rho_2(0) = 0$, then
$$
\rho_5(0) = 0,\ \ \{\rho_3(0), \rho_4(0)\} = \{\alpha, \beta\},
$$
and hence by the same argument we get
$\alpha + \beta = 0$. Let
$$
f_1(z) = \rho_2(z) + \rho_5(z)\ \ \&\ \ g_1(z) = f_1(\rho_2(z)) f_1(\rho_5(z)) = (z + \rho_4(z)) (z + \rho_3(z)).
$$
Then
$$g_1(0) = - \alpha^2 = f_1(0)^2 = 0,$$
this is a contradiction.

If $\rho_2(0)  \neq 0$, then $\{\rho_2(0), \rho_5(0)\} = \{\alpha, \beta\}$, $\{\rho_3(0), \rho_4(0)\} = \{0, 0\}$, and $\alpha + \beta = 0$. Similarly, we also have a contradiction. Thus the partition $\{\{0\}, \{1, 6\}, \{2, 5\}, \{3, 4\}\}$ is impossible.

(ii) If the partition is $\{\{0\}, \{1, 2, 4\}, \{3, 5, 6\}\}$ and $\dim \widetilde{\HA_\phi} \neq 1$. Recall that
$$\HL = \text{span}\left\{(a_1, \cdots, a_q): f \in D, \sum_{i=1}^q a_i F_i(0) = 0, \sum_{i=1}^q \overline{a_i} H_i(0) = 0\right\},$$
where
$$
F_i(z) = \sum\limits_{\rho \in G_i} f(\rho(z)) \rho(z)\ \ \&\ \ H_i(z) = \sum\limits_{\rho \in G_i^{-1}} f(\rho(z)) \rho(z).
$$
By Proposition \ref{vaiddatlms}, $\dim \HL \neq 1$. Since $\rho_0(0) = 0$, it follows that for $f \in D^2$, the equation
\begin{align*}
\sum_{i=1}^3 a_i F_i(0) &= a_2 [f(\rho_1(0)) \rho_1(0) + f(\rho_2(0)) \rho_2(0)  + f(\rho_4(0))\rho_4(0)] \\
&\hspace{0.5cm}+ a_3 [f(\rho_3(0))\rho_3(0) + f(\rho_5(0)) \rho_5(0) + f(\rho_6(0)) \rho_6(0)] \notag \\
& = 0 \notag
\end{align*}
has a nontrivial solution $(a_2, a_3, a_4)$. There are essentially two cases:
$$\{\rho_1(0), \rho_2(0), \rho_4(0)\} = \{0, 0, \alpha\}\ \ \hbox{or}\ \ \{\rho_1(0), \rho_2(0), \rho_4(0)\} = \{0, 0, 0\}.
$$

If $\{\rho_1(0), \rho_2(0), \rho_4(0)\} = \{0, 0, \alpha\}$, since $\alpha \neq \beta$, we see that $a_2 = a_3 = 0$.

If $\{\rho_1(0), \rho_2(0), \rho_4(0)\} = \{0, 0, 0\}$, then $\{\rho_3(0), \rho_5(0), \rho_6(0)\} = \{0, \alpha, \beta\}$, and $a_3 = 0$. Since
\begin{align*}
\sum_{i=1}^3 a_i H_i(0) &= \overline{a_2} [f(\rho_3(0))\rho_3(0) + f(\rho_5(0)) \rho_5(0) + f(\rho_6(0)) \rho_6(0)] \\
&\hspace{0.5cm}+ \overline{a_3} [f(\rho_1(0)) \rho_1(0) + f(\rho_2(0)) \rho_2(0)  + f(\rho_4(0))\rho_4(0)] \notag \\
& = 0, \notag
\end{align*}
we get $$\overline{a_2} [f(\alpha)\alpha + f(\beta)\beta] = 0.$$
Choosing $f = \varphi_\alpha$, we have $a_2 = 0$, thus $\dim \widetilde{\HA_\phi} = 1$ - this contradicts the assumption that $\dim \widetilde{\HA_\phi} \neq 1$.

Therefore $\dim \widetilde{\HA_\phi} = 1$, and $M_\phi$ is irreducible on $D^2$.

(d) $\phi = z^5 \varphi_\alpha^2, \alpha \neq 0$. We show that $\dim \widetilde{\HA_\phi} = 1$. Suppose on the contrary that $\dim \widetilde{\HA_\phi} \neq 1$, then the partition is $$\{\{0\}, \{1, 6\}, \{2, 5\}, \{3, 4\}\}\ \ \hbox{or}\ \ \{\{0\}, \{1, 2, 4\}, \{3, 5, 6\}\}.
$$
By the same argument in (b)-(i), we see that the partition $\{\{0\}, \{1, 6\}, \{2, 5\}, \{3, 4\}\}$ is impossible.

Suppose that the partition is $\{\{0\}, \{1, 2, 4\}, \{3, 5, 6\}\}$ and $\dim \widetilde{\HA_\phi} \neq 1$. For $f \in D^2$, assume that the equation
\begin{align*}
\sum_{i=1}^3 a_i F_i(0) &= a_2 [f(\rho_1(0)) \rho_1(0) + f(\rho_2(0)) \rho_2(0)  + f(\rho_4(0))\rho_4(0)] \\
&\hspace{0.5cm}+ a_3 [f(\rho_3(0))\rho_3(0) + f(\rho_5(0)) \rho_5(0) + f(\rho_6(0)) \rho_6(0)] \notag \\
& = 0 \notag
\end{align*}
has a nontrivial solution $(a_2, a_3)$. By the same argument in (b)-(ii), we have
$$
\{\rho_1(0), \rho_2(0), \rho_4(0)\} = \{0, 0, \alpha\} = \{\rho_3(0), \rho_5(0), \rho_6(0)\}.
$$
If
$$
\begin{cases}
f_1(z) = \rho_1(z) + \rho_2(z) + \rho_4(z);\\
g(z) = f_1(\rho_1(z)) + f_1(\rho_2(z)) + f_1(\rho_4(z)) = f_1(z) + 2 (\rho_3(z) + \rho_5(z) + \rho_6(z)),
\end{cases}
$$
then
$$g(0) =  3\alpha = 2f_1(0) + f_1(\alpha),$$
thus $f_1(\alpha) = \alpha$. Since $\rho_0(\alpha) = \alpha$, we get
$$
\begin{cases}
\{\rho_1(\alpha), \rho_2(\alpha), \rho_4(\alpha)\} = \{0, 0, \alpha\};\\
\{\rho_3(\alpha), \rho_5(\alpha), \rho_6(\alpha)\} = \{0, 0, 0\}.
\end{cases}
$$
If
$$
\begin{cases}
f_2(z) = \rho_3(z) + \rho_5(z) + \rho_6(z);\\
g_1(z) = f_2(\rho_3(z)) + f_2(\rho_5(z)) + f_2(\rho_6(z)) = f_2(z) + 2f_1(z),
\end{cases}
$$
then
$$g_1(0) = f_2(0) + 2 f_1(0) = 3 \alpha = 2f_2(0) + f_2(\alpha) = 2\alpha + 0,$$
this is a contradiction.

Therefore $\dim \widetilde{\HA_\phi} = 1$, and $M_\phi$ is irreducible on $D^2$.

(e) $\phi = z^4 \varphi_\alpha \varphi_\beta \varphi_\gamma, \alpha\beta\gamma \neq 0$. We show that $\dim \widetilde{\HA_\phi} = 1$. Suppose on the contrary that $\dim \widetilde{\HA_\phi} \neq 1$, then the partition is $\{\{0\}, \{1, 6\}, \{2, 5\}, \{3, 4\}\}$ or $\{\{0\}, \{1, 2, 4\}, \{3, 5, 6\}\}$.

(i) If the partition is $\{\{0\}, \{1, 6\}, \{2, 5\}, \{3, 4\}\}$ and $\dim \widetilde{\HA_\phi} \neq 1$, then by the same argument in (b)-(i), we get $\{\rho_1(0), \rho_6(0)\}, \{\rho_2(0), \rho_5(0)\}, \{\rho_3(0), \rho_4(0)\} \neq \{0, 0\}$. Thus
$$
\begin{cases}
\{\rho_1(0), \rho_6(0)\} = \{0, \alpha\}, \{\rho_2(0),\rho_5(0)\} = \{0, \beta\},\\
\{\rho_3(0), \rho_4(0)\} = \{0, \gamma\}.
\end{cases}
$$
Since $\dim \widetilde{\HA_\phi} \neq 1$, it follows that for $f \in D^2$, the equation
\begin{align*}
\sum_{i=1}^4 a_i F_i(0) &= a_2 f(\alpha) \alpha + a_3 f(\beta)\beta + a_4 f(\gamma)\gamma = 0
\end{align*}
has a nontrivial solution $(a_2, a_3, a_4)$. Therefore $\alpha, \beta, \gamma$ can not be mutually distinct. Without loss of generality, suppose $\alpha = \beta$. If
$$
\begin{cases}
f_1(z) = \rho_1(z) + \rho_6(z);\\
g(z)= f_1(\rho_1(z)) + f_1(\rho_6(z))= 2z + \rho_2(z) + \rho_5(z);\\
g_1(z)= f_1(\rho_2(z)) + f_1(\rho_3(z)) = f_1(z) + \rho_3(z) + \rho_4(z),
\end{cases}
$$
then
$$
\begin{cases}
g(0) = \alpha = f_1(0) + f_1(\alpha);\\
g_1(0) = \alpha + \gamma = f_1(0) + f_1(\alpha),
\end{cases}
$$
and hence $\gamma = 0$ -  a contradiction. So, if the partition is $\{\{0\}, \{1, 6\}, \{2, 5\}, \{3, 4\}\}$, then $\dim \widetilde{\HA_\phi} = 1$.

(ii) If the partition is $\{\{0\}, \{1, 2, 4\}, \{3, 5, 6\}\}$ and $\dim \widetilde{\HA_\phi} \neq 1$, then by the same argument in (b)-(ii), we have $\{\rho_1(0), \rho_2(0), \rho_4(0)\} = \{0, 0, 0\}$ is impossible. Thus $\{\rho_1(0), \rho_2(0), \rho_4(0)\} = \{0, 0, \alpha\}$, $\{\rho_3(0), \rho_5(0), \rho_6(0)\} = \{0, \beta, \gamma\}$. Since the equation
\begin{align*}
\sum_{i=1}^3 a_i F_i(0) &= a_2 f(\alpha)\alpha + a_3[f(\beta)\beta + f(\gamma)\gamma] = 0\quad\hbox{under}\quad f\in D^2
\end{align*}
has a nontrivial solution $(a_2, a_3)$, we get $\alpha = \beta = \gamma$, and $a_2 = - 2a_3$. Also for $f \in D^2$, the equation
\begin{align*}
\sum_{i=1}^3 \overline{a_i} H_i(0) &= \overline{a_2}2f(\alpha)\alpha  + \overline{a_3} f(\alpha) = 0
\end{align*}
is satisfied, hence $\overline{a_3} = -2\overline{a_2}$, so $a_2 = a_3 = 0$ contradicting the assumption $\dim \widetilde{\HA_\phi} \neq 1$.

Thus if the partition is $\{\{0\}, \{1, 2, 4\}, \{3, 5, 6\}\}$, then $\dim \widetilde{\HA_\phi} = 1$. By (i) and (ii), we have $\dim \widetilde{\HA_\phi} = 1$.

(f) $\phi = z^3 \varphi_{\alpha_1} \varphi_{\alpha_2} \varphi_{\alpha_3}\varphi_{\alpha_4}, \alpha_1\alpha_2\alpha_3\alpha_4 \neq 0$. We show that $\dim \widetilde{\HA_\phi} = 1$. Suppose on the contrary that $\dim \widetilde{\HA_\phi} \neq 1$, then the partition is
$$
\{\{0\}, \{1, 6\}, \{2, 5\}, \{3, 4\}\}\ \ \hbox{or}\ \ \{\{0\}, \{1, 2, 4\}, \{3, 5, 6\}\}.
$$

(i) If the partition is $\{\{0\}, \{1, 6\}, \{2, 5\}, \{3, 4\}\}$ and $\dim \widetilde{\HA_\phi} \neq 1$, then by the same argument in (b)-(i), we have $$\{\rho_1(0), \rho_6(0)\}, \{\rho_2(0), \rho_5(0)\}, \{\rho_3(0), \rho_4(0)\} \neq \{0, 0\}.
$$
Thus
$$
\begin{cases}
\{\rho_1(0), \rho_6(0)\} = \{0, \alpha_1\};\\
\{\rho_2(0), \rho_5(0)\} = \{0, \alpha_2\};\\
\{\rho_3(0), \rho_4(0)\} = \{\alpha_3, \alpha_4\}.
\end{cases}
$$
Next, we show that $\alpha_3 = \alpha_4$. If $\alpha_3 \neq \alpha_4$, since $\dim \widetilde{\HA_\phi} \neq 1$, it follows that for $f \in D^2$, the equation
\begin{align*}
\sum_{i=1}^4 a_i F_i(0) &= a_2 f(\alpha_1) \alpha_1 + a_3 f(\alpha_2)\alpha_2 + a_4 [f(\alpha_3)\alpha_3 + f(\alpha_4)\alpha_4] = 0 \notag
\end{align*}
has a nontrivial solution $(a_2, a_3, a_4)$. Thus $\alpha_1 = \alpha_2$ or $\{\alpha_1, \alpha_2\} = \{\alpha_3, \alpha_4\}$.

If $\alpha_1 = \alpha_2: = \alpha$, and
$$
\begin{cases}
f_1(z) = \rho_1(z) + \rho_6(z);\\
g(z) = f_1(\rho_2(z)) + f_1(\rho_5(z)) = f_1(z) + \rho_3(z) + \rho_4(z);\\
g_1(z) = f_1(\rho_1(z)) + f_1(\rho_6(z)) = 2z + \rho_2(z) + \rho_5(z),
\end{cases}
$$
then
$$
\begin{cases}
g(0) = \alpha + \alpha_3 + \alpha_4 = f_1(0) + f_1(\alpha);\\
g_1(0) = \alpha = f_1(0) + f_1(\alpha),
\end{cases}
$$
and hence $\alpha_3 + \alpha_4 = 0$. If
$$
\begin{cases}
f_2(z) = \rho_2(z) + \rho_5(z);\\
g_2(z) = f_2(\rho_1(z)) + f_2(\rho_6(z)) = f_1(z) + \rho_3(z) + \rho_4(z);\\
g_3(z) = f_2(\rho_2(z)) f_2(\rho_5(z)) = (z + \rho_4(z))(z + \rho_3(z)),
\end{cases}
$$
then
$$
\begin{cases}
g_2(0) = \alpha = f_2(0) + f_2(\alpha) = \alpha + f_2(\alpha);\\
g_3(0) = -\alpha_3^2 = f_2(0)f_2(\alpha),
\end{cases}
$$
but this is a contradiction.

If $\{\alpha_1, \alpha_2\} = \{\alpha_3, \alpha_4\}$, then setting
$$
\begin{cases}
f_3(z) = \rho_3(z) + \rho_4(z);\\
g_4(z) = f_3(\rho_3(z)) + f_3(\rho_4(z)) = 2 z + \rho_1(z) + \rho_6(z);\\
g_5(z) = f_3(\rho_1(z)) + f_3(\rho_6(z)) = \rho_4(z) + \rho_5(z) + \rho_2(z) + \rho_3(z);\\
g_6(z) = f_3(\rho_2(z)) + f_3(\rho_5(z)) = \rho_5(z) + \rho_6(z) + \rho_1(z) + \rho_2(z),
\end{cases}
$$
gives
$$
\begin{cases}
g_4(0) = \alpha_1 = f_3(\alpha_3) + f_3(\alpha_4) = f_3(\alpha_1) + f_3(\alpha_2);\\
g_5(0) = \alpha_2 + \alpha_3 + \alpha_4 = f_3(0) + f_3(\alpha_1) = \alpha_3 + \alpha_4 + f_3(\alpha_1);\\
g_6(0) = \alpha_1 + \alpha_2 = f_3(0) + f_3(\alpha_2) = \alpha_3 + \alpha_4 + f_3(\alpha_2).
\end{cases}
$$
Hence
$$
f_3(\alpha_1) + f_3(\alpha_2) = \alpha_1 = \alpha_2,
$$
which yields a contradiction when $\alpha_1 = \alpha_2$. So $\alpha_3 = \alpha_4: = \alpha$.

Since for $f \in D^2$ the equation
\begin{align*}
\sum_{i=1}^4 a_i F_i(0) &= a_2 f(\alpha_1)\alpha_1 + a_3f(\alpha_2)\alpha_2  + a_42f(\alpha)\alpha = 0
\end{align*}
has a nontrivial solution $(a_2, a_3, a_4)$, we obtain that $\alpha_1, \alpha_2, \alpha$ are not mutually distinct.

If $\alpha_1 = \alpha_2$, we have shown that this is impossible.

If $\alpha_1 = \alpha$, let $f_1, g_1$ be defined as above, i.e.,
$$
\begin{cases}
f_1(z) = \rho_1(z) + \rho_6(z), g_1(z) = f_1(\rho_1(z)) + f_1(\rho_6(z));\\
g_7(z) = f_1(\rho_3(z)) + f_1(\rho_4(z)) = \rho_4(z) + \rho_5(z) + \rho_2(z) + \rho_3(z),
\end{cases}
$$
then
$$
\begin{cases}
g_1(0) = \alpha_2 = f_1(0) + f_1(\alpha_1) = \alpha_1 + f_1(\alpha_1);\\
g_7(0) = \alpha_2 + 2\alpha_1 = 2f_1(\alpha_1).
\end{cases}
$$
Hence $\alpha_2 = 4 \alpha_1$, $f_1(\alpha_1) = 3 \alpha_1$, this is impossible.

Similarly, $\alpha_2 = \alpha$ is also impossible. Thus if the partition is $\{\{0\}, \{1, 6\}, \{2, 5\}, \{3, 4\}\}$, then $\dim \widetilde{\HA_\phi} = 1$.

(ii) If the partition is $\{\{0\}, \{1, 2, 4\}, \{3, 5, 6\}\}$ and $\dim \widetilde{\HA_\phi} \neq 1$, there are essentially two cases: $\{\rho_1(0), \rho_2(0), \rho_4(0)\} = \{0, 0, \alpha_1\}$ or $\{\rho_1(0), \rho_2(0), \rho_4(0)\} = \{0, \alpha_1, \alpha_2\}$.

If
$$
\begin{cases}\{\rho_1(0), \rho_2(0), \rho_4(0)\} = \{0, 0, \alpha_1\};\\
\{\rho_3(0), \rho_5(0), \rho_6(0)\} = \{\alpha_2, \alpha_3, \alpha_4\},
\end{cases}
$$
then application of the fact that for $f \in D^2$ the equation
\begin{align*}
\sum_{i=1}^3 a_i F_i(0) &= a_2 f(\alpha_1)\alpha_1 + a_3[f(\alpha_2)\alpha_2 + f(\alpha_3)\alpha_3 + f(\alpha_4)\alpha_4] = 0
\end{align*}
has a nontrivial solution $(a_2, a_3)$, derives
$$
\alpha_1 = \alpha_2 = \alpha_3 = \alpha_4\quad\&\quad a_2 = -3a_3.
$$
Also for $f \in D^2$ the equation
\begin{align*}
\sum_{i=1}^3 \overline{a_i} H_i(0) &= \overline{a_2} 3f(\alpha_1)\alpha_1 + \overline{a_3} f(\alpha_1)\alpha_1= 0
\end{align*}
is satisfied, thus
$$
\overline{a_3} = -3 \overline{a_2}\ \ \&\ \ a_2 = a_3 = 0,
$$
this is against the assumption $\dim \widetilde{\HA_\phi} \neq 1$.

If
$$
\begin{cases}
\{\rho_1(0), \rho_2(0), \rho_4(0)\} = \{0, \alpha_1, \alpha_2\};\\
\{\rho_3(0), \rho_5(0), \rho_6(0)\} = \{0, \alpha_3, \alpha_4\},
\end{cases}
$$
then using $\dim \widetilde{\HA_\phi} \neq 1$ we get
$$
\{\alpha_1, \alpha_2\} = \{\alpha_3, \alpha_4\}.
$$
Let
$$
\begin{cases}
f_1(z) = \rho_1(z) + \rho_2(z) + \rho_4(z);\\
f_2(z) = \rho_3(z) + \rho_5(z) + \rho_6(z);\\
g(z) = f_1(\rho_1(z)) + f_1(\rho_2(z)) + f_1(\rho_4(z)) = f_1(z) + 2 f_2(z);\\
g_1(z) = f_1(\rho_3(z)) + f_1(\rho_5(z)) + f_1(\rho_6(z)) = 3z + f_1(z) + f_2(z).
\end{cases}
$$
Then
$$
\begin{cases}
g(0) = f_1(0) + 2(\alpha_1 + \alpha_2) = f_1(0) + f_1(\alpha_1) + f_1(\alpha_2);\\
g_1(0) = f_1(0) + \alpha_1 + \alpha_2 = f_1(0) + f_1(\alpha_1) + f_1(\alpha_2).
\end{cases}
$$
Accordingly,
$$
\begin{cases}
\alpha_1 + \alpha_2 = 0;\\
f_1(\alpha_1) + f_1(\alpha_2) = 0;\\
f_1(-\alpha_1) = - f_1(\alpha_1).
\end{cases}
$$
Note that
$$
\phi^{-1}\circ \phi (\alpha_1) = \{0, 0, 0, \alpha_1, \alpha_1, - \alpha_1, -\alpha_1\}.
$$
So we have the following six cases:
$$
\{\rho_1(\alpha_1), \rho_2(\alpha_1), \rho_4(\alpha_1)\} =\begin{cases} \{0, 0, 0\};\\
\{0, 0, \alpha_1\};\\
\{0, 0, -\alpha_1\};\\
\{0, \alpha_1, -\alpha_1\};\\
\{0, -\alpha_1, -\alpha_1\};\\
\{\alpha_1, -\alpha_1, \alpha_1\}.
\end{cases}
$$

If
$$
\{\rho_1(\alpha_1), \rho_2(\alpha_1), \rho_4(\alpha_1)\} = \{0, 0, 0\},
$$
then
$$
\{\rho_3(\alpha_1), \rho_5(\alpha_1), \rho_6(\alpha_1)\} = \{\alpha_1, -\alpha_1, -\alpha_1\},
$$
and hence $g(\alpha_1) = 0 - 2\alpha_1 = 3f(0) = 0$ - this is a contradiction. Similarly, using $g(\alpha_1)$ and $f_1(-\alpha_1) = - f_1(\alpha_1)$, we get contradictions for the rest cases.

Therefore if the partition is $\{\{0\}, \{1, 2, 4\}, \{3, 5, 6\}\}$, then $\dim \widetilde{\HA_\phi} = 1$. By (i) and (ii), we have $\dim \widetilde{\HA_\phi} = 1$.

(g) $\phi = z^2 \varphi_{\alpha_1} \varphi_{\alpha_2} \varphi_{\alpha_3}\varphi_{\alpha_4}\varphi_{\alpha_5}, \alpha_1\alpha_2\alpha_3\alpha_4\alpha_5 \neq 0$. We show that $\dim \widetilde{\HA_\phi} = 1$. Suppose on the contrary that $\dim \widetilde{\HA_\phi} \neq 1$, then the partition is
$$\{\{0\}, \{1, 6\}, \{2, 5\}, \{3, 4\}\}\ \ \hbox{or}\ \ \{\{0\}, \{1, 2, 4\}, \{3, 5, 6\}\}.
$$

(i) If the partition is $\{\{0\}, \{1, 6\}, \{2, 5\}, \{3, 4\}\}$ and $\dim \widetilde{\HA_\phi} \neq 1$, suppose
$$
\begin{cases}
\{\rho_1(0), \rho_6(0)\} = \{0, \alpha_1\};\\
\{\rho_2(0), \rho_5(0)\} = \{\alpha_2, \alpha_3\};\\
\{\rho_3(0), \rho_4(0)\} = \{\alpha_4, \alpha_5\},
\end{cases}
$$
we show that $\{\alpha_2, \alpha_3\} \neq \{\alpha_4, \alpha_5\}$. If $\{\alpha_2, \alpha_3\} = \{\alpha_4, \alpha_5\}$, then setting
$$
\begin{cases}
f_1(z) = \rho_2(z) + \rho_5(z);\\
g(z) = f_1(\rho_3(z)) + f_1(\rho_4(z)) = f_1(z) + \rho_1(z) + \rho_6(z);\\
g_1(z) = f_1(\rho_2(z)) + f_1(\rho_5(z)) = 2z + \rho_3(z) + \rho_4(z)
\end{cases}
$$
gives
$$
\begin{cases}
g(0) = \alpha_2 + \alpha_3 + \alpha_1 = f_1(\alpha_2) + f_1(\alpha_3);\\
g_1(0) = \alpha_2 + \alpha_3 = f_1(\alpha_2) + f_1(\alpha_3),
\end{cases}
$$
hence $\alpha_1 = 0$ - this is a contradiction. So $\{\alpha_2, \alpha_3\} \neq \{\alpha_4, \alpha_5\}$.

Now we show $\alpha_2 \neq \alpha_3$. If $\alpha_2 = \alpha_3: = \alpha$, without loss of generality, suppose $\alpha_4 \neq \alpha$. We show $\alpha_5 \neq \alpha$. If $\alpha_5 = \alpha$, since $\dim \widetilde{\HA_\phi} \neq 1$, it follows that for $f \in D^2$, the equation
\begin{align*}
\sum_{i=1}^4 a_i F_i(0) &= a_2 f(\alpha_1) \alpha_1 + a_3 2f(\alpha)\alpha + a_4 [f(\alpha_4)\alpha_4 + f(\alpha_5)\alpha_5] = 0
\end{align*}
has a nontrivial solution $(a_2, a_3, a_4)$. So $\alpha_1 = \alpha_4$ or $\alpha_1 = \alpha$. If $\alpha_1 = \alpha_4$, then letting
$$
\begin{cases}
f_2(z) = \rho_3(z) + \rho_4(z);\\
g_2(z) = f_2(\rho_3(z)) + f_2(\rho_4(z)) = 2z + \rho_1(z) + \rho_6(z);\\
g_3(z) = f_2(\rho_1(z)) + f_2(\rho_6(z)) = \rho_4(z) + \rho_5(z) + \rho_2(z) + \rho_3(z);\\
g_4(z) = f_2(\rho_2(z)) + f_2(\rho_5(z)) = \rho_5(z) + \rho_6(z) + \rho_1(z) + \rho_2(z),
\end{cases}
$$
derives
$$
\begin{cases}
g_2(0) = \alpha_1 = f_2(\alpha_1) + f_2(\alpha);\\
g_3(0) = 3\alpha + \alpha_1 = f_2(0) + f_2(\alpha_1) = \alpha_1 + \alpha + f_2(\alpha_1);\\
g_4(0) = \alpha_1 + 2\alpha = 2f_2(\alpha),
\end{cases}
$$
and consequently,
$$
\begin{cases}
f_2(\alpha_1) = 2\alpha;\\
f_2(\alpha) = \alpha_1 - 2\alpha;\\
\alpha_1 = 6\alpha.
\end{cases}
$$
Note that
$$
\phi^{-1}\circ\phi(\alpha) = \{0, 0, \alpha_1, \alpha_1, \alpha, \alpha, \alpha\}.
$$
So we have
$$
f_2(\alpha) = 0, \alpha_1, \alpha, 2\alpha_1, 2\alpha, \alpha_1 + \alpha.
$$
But none of these satisfies
$$
f_2(\alpha) = \alpha_1 - 2\alpha\ \ \&\ \ \alpha_1 = 6\alpha.
$$
This is a contradiction.

If $\alpha_1 = \alpha$, then by the argument below, we also have a contradiction.

Thus $\alpha_5 \neq \alpha$, then by $\dim \widetilde{\HA_\phi} \neq 1$, we get $\alpha_1 = \alpha$ or $\alpha_1 = \alpha_4 = \alpha_5$. Without loss of generality, suppose $\alpha_1 = \alpha$. Let
$$
\begin{cases}
f_3(z) = \rho_2(z) + \rho_5(z);\\
g_5(z) = f_3(\rho_2(z)) + f_3(\rho_5(z)) = 2z + \rho_3(z) + \rho_4(z);\\
g_6(z) = f_3(\rho_1(z)) + f_3(\rho_6(z)) = \rho_1(z) + \rho_6(z) + \rho_3(z) + \rho_4(z).
\end{cases}
$$
Then
$$
\begin{cases}
g_5(0) = \alpha_4 + \alpha_5 = 2f_3(\alpha);\\
g_6(0) = \alpha_4 + \alpha_5 + \alpha = f_3(0) + f_3(\alpha) = 2\alpha + f_3(\alpha);\\
f_3(\alpha) = \alpha;\\
\alpha_4 + \alpha_5 = 2\alpha.
\end{cases}
$$
Let
$$
\begin{cases}
f_4(z) = \rho_1(z) + \rho_6(z);\\
g_7(z) = f_4(\rho_1(z)) + f_4(\rho_6(z)) = 2z + \rho_2(z) + \rho_5(z);\\
g_8(z) = f_4(\rho_2(z)) + f_4(\rho_5(z)) = \rho_1(z) + \rho_6(z) + \rho_3(z) + \rho_4(z).
\end{cases}
$$
Then
$$
\begin{cases}
g_7(0) = 2 \alpha = f_4(0) + f_3(\alpha) = \alpha + f_4(\alpha);\\
g_8(0) = \alpha_4 + \alpha_5 + \alpha = 2 f_4(\alpha);\\
f_4(\alpha) = \alpha;\\
\alpha_4 + \alpha_5 = \alpha.
\end{cases}
$$
However the last equations contain a contradiction. Therefore $\alpha_2 \neq \alpha_3$.

We have shown $\{\alpha_2, \alpha_3\} \neq \{\alpha_4, \alpha_5\}$, without loss of generality, suppose $\alpha_4 \not\in \{\alpha_2, \alpha_3\}$. Since $\dim \widetilde{\HA_\phi} \neq 1$, it follows that for $f \in D^2$, the equation
\begin{equation}\label{szs25e}
\sum_{i=1}^4 a_i F_i(0)= a_2 f(\alpha_1) \alpha_1 + a_3 [f(\alpha_2)\alpha_2 + f(\alpha_3)\alpha_3] + a_4 [f(\alpha_4)\alpha_4 + f(\alpha_5)\alpha_5]= 0
\end{equation}
has a nontrivial solution $(a_2, a_3, a_4)$. Next, we prove $\alpha_1 \neq \alpha_4$. If not, then using the fact that $\alpha_1 = \alpha_4 = \alpha_5$ is impossible we get $\alpha_5\not=\alpha_1$. Upon choosing $f = \varphi_{\alpha_2}\varphi_{\alpha_3}\varphi_{\alpha_5}$ in (\ref{szs25e}), we have
$$a_2f(\alpha_1) \alpha_1 + a_4 f(\alpha_1) \alpha_1 = 0,$$
whence $a_2 = -a_4$, and so (\ref{szs25e}) becomes
$$a_3 [f(\alpha_2)\alpha_2 + f(\alpha_3)\alpha_3] + a_4 f(\alpha_5)\alpha_5 = 0.$$
Thus $a_3 = a_4 = 0$, so $a_2 = a_3 = a_4 = 0$ - this is a contradiction. Therefore $\alpha_1 \neq \alpha_4$, similarly, $\alpha_1 \neq \alpha_5$. If $\alpha_4 = \alpha_5$, then taking $f = \varphi_{\alpha_1}\varphi_{\alpha_2}\varphi_{\alpha_3}$ in (\ref{szs25e}) yields $a_4 = 0$, then $a_2 = a_3 = 0$, a contradiction. If $\alpha_4 \neq \alpha_5$, then choosing $f = \varphi_{\alpha_1}\varphi_{\alpha_2}\varphi_{\alpha_3}\varphi_{\alpha_5}$ in (\ref{szs25e}) gives $a_4 = 0$, and then $a_2 = a_3 = 0$, a contradiction.

Therefore if $\{\alpha_2, \alpha_3\} \neq \{\alpha_4, \alpha_5\}$, we have a contradiction. So if the partition is
$$
\{\{0\}, \{1, 6\}, \{2, 5\}, \{3, 4\}\},
$$
then $\dim \widetilde{\HA_\phi} = 1$.

(ii) If the partition is
$\{\{0\}, \{1, 2, 4\}, \{3, 5, 6\}\}$ and $\dim \widetilde{\HA_\phi} \neq 1,
$
suppose
$$
\begin{cases}
\{\rho_1(0), \rho_2(0), \rho_4(0)\} = \{0, \alpha_1, \alpha_2\};\\
\{\rho_3(0), \rho_5(0), \rho_6(0)\} = \{\alpha_3, \alpha_4, \alpha_5\}.
\end{cases}
$$
 We show $\alpha_1 = \alpha_2$. Since $\dim \widetilde{\HA_\phi} \neq 1$, for $f \in D^2$ the equation
\begin{align*}
\sum_{i=1}^3 a_i F_i(0) &= a_2 [f(\alpha_1) \alpha_1 + f(\alpha_2) \alpha_2] + a_3 [f(\alpha_3) \alpha_3 + f(\alpha_4)\alpha_4 + f(\alpha_5)\alpha_5] = 0,
\end{align*}
has a nontrivial solution $(a_2, a_3)$. If $\alpha_1 \neq \alpha_2$, then it is straightforward to deduce $a_2 = a_3 = 0$ - this is a contradiction.

Thus $\alpha_1 = \alpha_2$. Since for $f \in D^2$ both
\begin{align*}
\sum_{i=1}^3 a_i F_i(0) &= a_2 [f(\alpha_1) \alpha_1 + f(\alpha_2) \alpha_2] + a_3 [f(\alpha_3) \alpha_3 + f(\alpha_4)\alpha_4 + f(\alpha_5)\alpha_5] = 0
\end{align*}
and
\begin{align*}
\sum_{i=1}^3 \overline{a_i} H_i(0) &= \overline{a_2}[f(\alpha_3) \alpha_3 + f(\alpha_4)\alpha_4 + f(\alpha_5)\alpha_5] + \overline{a_3} [f(\alpha_1) \alpha_1 + f(\alpha_2) \alpha_2]  = 0
\end{align*}
have a nontrivial solution $(a_2, a_3)$, by the same argument in (e)-(ii) we get $$\alpha_1 = \alpha_2 = \alpha_3 = \alpha_4 = \alpha_5,
$$
thereby finding $a_2 = a_3 = 0$, a contradiction.

Therefore if the partition is $\{\{0\}, \{1, 2, 4\}, \{3, 5, 6\}\}$, then $\dim \widetilde{\HA_\phi} = 1$. By (i) and (ii), we have $\dim \widetilde{\HA_\phi} = 1$.

(h) $\phi = z \varphi_{\alpha_1} \varphi_{\alpha_2} \varphi_{\alpha_3}\varphi_{\alpha_4}\varphi_{\alpha_5}\varphi_{\alpha_6}, \alpha_1\alpha_2\alpha_3\alpha_4\alpha_5\alpha_6 \neq 0$. We have the following cases.

(i) If the partition is $\{\{0\}, \{1\}, \{2\}, \{3\}, \{4\}, \{5\}, \{6\}\}$ or $\{\{0\}, \{1, 2, 3, 4, 5, 6\}\}$, then by the argument in Theorem \ref{rdopo5ubp1}, we have $M_\phi$ is irreducible on $D^2$.

(ii) If the partition is $\{\{0\}, \{1, 6\}, \{2, 5\}, \{3, 4\}\}$, we show $\dim \widetilde{\HA_\phi} = 1$. If $\dim \widetilde{\HA_\phi} \neq 1$, suppose
$$
\begin{cases}
\{\rho_1(0), \rho_6(0)\} = \{\alpha_1, \alpha_2\};\\
\{\rho_2(0), \rho_5(0)\} = \{\alpha_3, \alpha_4\};\\
\{\rho_3(0), \rho_4(0)\} = \{\alpha_5, \alpha_6\}.
\end{cases}
$$
We show that $\{\alpha_1, \alpha_2\}, \{\alpha_3, \alpha_4\}, \{\alpha_5, \alpha_6\}$ are mutually distinct. If $\{\alpha_1, \alpha_2\} = \{\alpha_3, \alpha_4\}$, then letting
$$
\begin{cases}
f_1(z) = \rho_1(z) \rho_6(z);\\
g(z) = f_1(\rho_1(z)) f_1(\rho_6(z)) = z^2 \rho_2(z) \rho_5(z);\\
g_1(z) = f_1(\rho_2(z)) f_1(\rho_5(z)) = \rho_1(z) \rho_6(z) \rho_3(z) \rho_4(z),
\end{cases}
$$
gives
$$
\begin{cases}
g(0) = 0 = f_1(\alpha_1) f_1(\alpha_2);\\
g_1(0) = \alpha_1 \alpha_2 \alpha_5 \alpha_6 = f_1(\alpha_1) f_1(\alpha_2),
\end{cases}
$$
thereby yielding a contradiction. Thus $\{\alpha_1, \alpha_2\}, \{\alpha_3, \alpha_4\}$ are distinct. Similarly, we have $\{\alpha_1, \alpha_2\}, \{\alpha_3, \alpha_4\}, \{\alpha_5, \alpha_6\}$ are mutually distinct.

Now we show that not both $\alpha_1$ and $\alpha_2$ are in $\{\alpha_3, \alpha_4, \alpha_5, \alpha_6\}$. Let
$$
\begin{cases}
f_2(z) = \rho_1(z) \rho_6(z);\\
g_2(z) = f_2(\rho_1(z)) f_2(\rho_6(z)) = z^2 \rho_2(z) \rho_5(z);\\
g_3(z) = f_2(\rho_3(z)) f_2(\rho_4(z)) = \rho_2(z) \rho_3(z) \rho_4(z) \rho_5(z);\\
g_4(z) = f_2(\rho_2(z)) f_2(\rho_5(z)) = \rho_1(z) \rho_3(z) \rho_4(z) \rho_6(z).
\end{cases}
$$
Then
$$
\begin{cases}
g_2(0) = 0 = f_2(\alpha_1) f_2(\alpha_2);\\
g_3(0) = \alpha_3 \alpha_4 \alpha_5 \alpha_6 = f_2(\alpha_5) f_2(\alpha_6);\\
g_4(0) = \alpha_1 \alpha_2 \alpha_5 \alpha_6 = f_2(\alpha_3) f_2(\alpha_4).
\end{cases}
$$
it follows that either $\alpha_1$ or $\alpha_2$ is not in $\{\alpha_3, \alpha_4, \alpha_5, \alpha_6\}$.

If $\alpha_1 = \alpha_2$, then choosing $f = \varphi_{\alpha_3}\varphi_{\alpha_4}\varphi_{\alpha_5}\varphi_{\alpha_6}$ in the following equation
\begin{align*}
\sum_{i=1}^4 a_i F_i(0)= a_2 2f(\alpha_1) \alpha_1 + a_3 [f(\alpha_3)\alpha_3 + f(\alpha_4)\alpha_4] + a_4 [f(\alpha_5)\alpha_5 + f(\alpha_6)\alpha_6] = 0\notag
\end{align*}
we get $a_2 = 0$, since $\{\alpha_3, \alpha_4\} \neq \{\alpha_5, \alpha_6\}$, we obtain $a_3 = a_4 = 0$ - this is a contradiction.

If $\alpha_1 \neq \alpha_2$, without loss of generality, suppose $\alpha_1 \not\in \{\alpha_3, \alpha_4, \alpha_5, \alpha_6\}$, then taking
$$f = \varphi_{\alpha_2}\varphi_{\alpha_3}\varphi_{\alpha_4}\varphi_{\alpha_5}\varphi_{\alpha_6}$$ in the equation
\begin{align*}
a_2 [f(\alpha_1) \alpha_1 + f(\alpha_2) \alpha_2] + a_3 [f(\alpha_3)\alpha_3 + f(\alpha_4)\alpha_4] + a_4 [f(\alpha_5)\alpha_5 + f(\alpha_6)\alpha_6] = 0,
\end{align*}
we get $a_2 = 0$, it follows that $a_3 = a_4 = 0$ - this is also a contradiction. Thus if the partition is $\{\{0\}, \{1, 6\}, \{2, 5\}, \{3, 4\}\}$, then $\dim \widetilde{\HA_\phi} = 1$.

(iii) If the partition is $\{\{0\}, \{1, 2, 4\}, \{3, 5, 6\}\}$, we show $\dim \widetilde{\HA_\phi} = 1$. If $\dim \widetilde{\HA_\phi} \neq 1$, suppose
$$
\begin{cases}
\{\rho_1(0), \rho_2(0), \rho_4(0)\} = \{\alpha_1, \alpha_2, \alpha_3\};\\
\{\rho_3(0), \rho_5(0), \rho_6(0)\} = \{\alpha_4, \alpha_5, \alpha_6\}.
\end{cases}
$$
 Since $\dim \widetilde{\HA_\phi} \neq 1$, for $f \in D^2$ the equation
\begin{align*}
\sum_{i=1}^3 a_i F_i(0) &= a_2 [f(\alpha_1) \alpha_1 + f(\alpha_2) \alpha_2 + f(\alpha_3) \alpha_3] + a_3 [f(\alpha_4) \alpha_4 + f(\alpha_4)\alpha_4 + f(\alpha_5)\alpha_5] = 0
\end{align*}
has a nontrivial solution $(a_2, a_3)$. Hence $\{\alpha_1, \alpha_2, \alpha_3\} = \{\alpha_4, \alpha_5, \alpha_6\}$. Upon setting
$$
\begin{cases}
f_1(z) = \rho_1(z) \rho_2(z) \rho_4(z);\\
g(z) = f_1(\rho_1(z)) f_1(\rho_2(z)) f_1(\rho_4(z)) = f_1(z) (\rho_3(z) \rho_5(z) \rho_6(z))^2;\\
g_1(z) = f_1(\rho_3(z)) f_1(\rho_5(z)) f_1(\rho_6(z)) = z^3 f_1(z) \rho_3(z) \rho_5(z) \rho_6(z),
\end{cases}
$$
we get
$$
\begin{cases}
g(0) = (\alpha_1 \alpha_2 \alpha_3)^3 = f_1(\alpha_1) f_1(\alpha_2) f_1(\alpha_3);\\
g_1(0) = 0 = f_1(\alpha_1) f_1(\alpha_2) f_1(\alpha_3),
\end{cases}
$$
thereby reaching a contradiction. Thus, if the partition is $\{\{0\}, \{1, 2, 4\}, \{3, 5, 6\}\}$, then $\dim \widetilde{\HA_\phi} = 1$. By (i), (ii) and (iii), we have $\dim \widetilde{\HA_\phi} = 1$. The proof is complete.
\end{proof}

\section{Concluding remarks}\label{s6}

On the one hand, we have seen from \S \ref{s3} \& \S \ref{s5} that if the order $n$ of $\phi$ is $5$ or $7$ then $M_\phi$ is reducible on $D^2$ if and only if $\phi$ is equivalent to $z^n$. This naturally leads to such a question that if the order of $\phi$ is prime then is it true that $M_\phi$ is reducible on $D^2$ if and only if $\phi$ is equivalent to $z^n$?

On the other hand, \S 4 reveals that if the order $n$ of $\phi$ is $6$ then $M_{\phi}$ is reducible on $D^{2}$ not only when $\phi$ is equivalent to
$z^{6}$ but also when $\phi$ is equivalent to many other functions. We believe the arguments in \S \ref{s4} can go through to the case when the order $n$ of $\phi$ is not prime, at least when $n$ is 8, 9, 10 or 12, since we know the partitions for $\phi$ when $n \leq 5$. When $n$ is 8, we find two comments on the reducing subspaces of $M_\phi$ on $B^2$ as follows.

\begin{itemize}
\item[(i)] The possible partitions $\{G_1, G_2, \cdots, G_q\}$ are discussed in Theorem 4.2 \cite{DPW12}. However, there is an error in item (2) of Theorem 4.2 \cite{DPW12}: when $\phi = \varphi_\alpha^2(\varphi_\beta^4), \alpha \neq 0$, the authors claimed that the partition for $\phi$ is
$$
\{\{0\}, \{1, 5\}, \{2\}, \{3, 7\}, \{4\}, \{6\}\}.
$$
Nevertheless, we point out that if $\phi = \varphi_\alpha^2(\varphi_\beta^4), \alpha \neq 0$ then the partition should be $\{\{0\}, \{1, 3, 5, 7\}, \{2\}, \{4\}, \{6\}\}$. Because we have
$$
\rho_{2j}(z) = \varphi_\beta(\omega_4^j\varphi_\beta(z)), j = 0, 1, 2, 3,
$$
where $\omega_4 = i$, and $\rho_k(z), k = 1, 3, 5, 7$ are solutions for $\varphi_\alpha(w^4) = - \varphi_\alpha(\varphi_\beta^4(z))$, in this case the Riemann surface $S_\phi$ has five connected components, and so the partition is $$\{\{0\}, \{1, 3, 5, 7\}, \{2\}, \{4\}, \{6\}\}.
$$
In fact, the dual partition for
$$
\{\{0\}, \{1, 3, 5, 7\}, \{2\}, \{4\}, \{6\}\}
$$
is
$$\{\{0\}, \{1, 5\}, \{2, 6\}, \{3, 7\}, \{4\}\}.
$$
Thus $$\{\{0\}, \{1, 3, 5, 7\},\{2\}, \{4\}, \{6\}\}$$ is also a possible partition for $\phi = \varphi_\alpha^2(\varphi_\beta^4), \alpha \neq 0$. And by the above reasoning, this one is the partition for $\phi = \varphi_\alpha^2(\varphi_\beta^4), \alpha \neq 0$.

\item[(ii)] When $\phi = \varphi_\beta^2(\varphi_\alpha^2(z^2))), \alpha \neq 0$, there are two cases. If $\beta \neq 0$, then by a little calculation, we obtain that the Riemann surface $S_\phi$ has four connected components, thus the partition is $\{\{0\}, \{1, 3, 5, 7\}, \{2, 6\}, \{4\}\}$. If $\beta  = 0$, Note that
\begin{align*}
\phi(w) - \phi(z)&= -\varphi_\alpha^4(w^2) + \varphi_\alpha^4(z^2)\\
& = -\prod_{j=0}^3[\varphi_\alpha(w^2) - i^j\varphi_\alpha(z^2)]\\
& = \frac{1 - |\alpha|^2}{(1- \overline{\alpha}w^2)(1- \overline{\alpha}z^2)}(w - z) (w+z) \prod_{j=1}^3[\varphi_\alpha(w^2) - i^j\varphi_\alpha(z^2)].
\end{align*}
So $\rho_4(z) = -z$. If $\rho_k(z), \rho_l(z)$ are the solutions for $\varphi_\alpha(w^2) - i^j\varphi_\alpha(z^2) = 0$ on $\Omega = u^{-1}(A_s)$, then the same holds for $\rho_k (\rho_4(z)), \rho_l (\rho_4(z))$, thus the partition is $\{\{0\}, \{1, 5\}, \{2, 6\}, \{3, 7\}, \{4\}\}$. This answers a question on page 1761 \cite{DPW12} and clarifies item (4) of Theorem 4.2 \cite{DPW12}.
\end{itemize}

\vspace{1cm}
\noindent \textbf{Acknowledgements.}
The authors thank Kai Wang for some helpful discussions on the dual partition, and the referee for her/his comments improving the readability of this paper.

\end{document}